\documentclass[12pt,a4paper]{amsart}%
\usepackage{imakeidx}
\usepackage{amsmath}
\usepackage{fullpage}
\usepackage{comment}
\usepackage{MyCommA}
\usepackage[final]{pdfpages}
\usepackage{tikz-cd}
\makeindex

\newcommand{\Rami}[1]{{{#1}}}

\newcommand{\Dima}[1]{{{#1}}}
\newcommand{\RamiA}[1]{{{#1}}}
\newcommand{\RamiB}[1]{{{#1}}}
\newcommand{\RamiC}[1]{{{#1}}}
\newcommand{\RamiD}[1]{{{#1}}}

\newcommand{\RamiE}[1]{{{#1}}}
\newcommand{\RamiF}[1]{{{#1}}}
\newcommand{\RamiG}[1]{{{#1}}}
\newcommand{\RamiH}[1]{{{#1}}}
\newcommand{\RamiI}[1]{{{#1}}}
\newcommand{\RamiJ}[1]{{{#1}}}
\newcommand{\RamiK}[1]{{{#1}}}
\newcommand{\EitanH}[1]{{{#1}}}
\newcommand{\typos}[1]{{{#1}}}
\newcommand{\DimaA}[1]{{{#1}}}

\newcommand{\NextVer}[1]{}
\newcommand{\sub}{\subset}
\newcommand{\noleft}{\left.\kern-\nulldelimiterspace}

\renewcommand{\val}{\mathrm{val}}
\newcommand{\Fk}{\F_{\ell^k}((t))}
\newcommand{\Ok}{\F_{\ell^k}[[t]]}

\makeatletter
\renewcommand{\p@enumii}{}               
\makeatother

\begin{document}
	
	\author[Aizenbud]{Avraham Aizenbud}
	\address{Avraham Aizenbud,
		Faculty of Mathematical Sciences,
		Weizmann Institute of Science,
		76100
		Rehovot, Israel}
	\email{aizenr@gmail.com}
	\urladdr{https://www.wisdom.weizmann.ac.il/~aizenr/}
	
	\author[Gourevitch]{Dmitry Gourevitch}
	\address{Dmitry Gourevitch,
		Faculty of Mathematical Sciences,
		Weizmann Institute of Science,
		76100
		Rehovot, Israel}
	\email{dimagur@weizmann.ac.il}
	\urladdr{https://www.wisdom.weizmann.ac.il/~dimagur/}

\author[Kazhdan]{David Kazhdan}
	\address{David Kazhdan,
    Einstein Institute of Mathematics, Edmond J. Safra Campus, Givaat Ram The
Hebrew University of Jerusalem, Jerusalem, 91904, Israel}
	\email{david.kazhdan@mail.huji.ac.il}
	\urladdr{https://math.huji.ac.il/~kazhdan/}
	
\author[Sayag]{Eitan Sayag}
	\address{Eitan Sayag,
    Department of Mathematics, Ben Gurion University of the Negev, P.O.B. 653,
Be’er Sheva 84105, ISRAEL}
	\email{eitan.sayag@gmail.com}
	\urladdr{www.math.bgu.ac.il/~sayage}

	\date{\today}
	
	\keywords{implicit function theorem, smooth measures, algebraic varieties in positive characteristic}
	\subjclass{54E35, 14G20, 26B10, 28C15}
	%
	%
	%
	%
	%
	%
	%
	%


    

\title{effective local differential topology of algebraic varieties over local fields of positive \RamiE{characteristic}}
\maketitle
\begin{abstract} 
In this paper we provide a framework for quantitative statements on distances and measures when studying algebraic varieties and morphisms of algebraic varieties over local fields. 

We will concentrate on local fields of the type $\F_\ell((t))$ and work uniformly with respect to finite extensions of $\F_\ell$.

In this framework we prove analogues of standard results from local differential topology, including the implicit function theorem and study the behavior of smooth measures under push forward {with respect to} submersions.
\end{abstract}

\tableofcontents 
\section{Introduction}
The goal of this paper is to provide {a} framework for formulating quantitative statements on distances and measures when studying algebraic varieties and maps between them over local fields.

We will concentrate on local fields of the type $\F_\ell((t))$ and work uniformly with respect to finite extensions of $\F_\ell$. \RamiE{Here $\ell$ is a fixed prime power.}

We introduce a notion of a rectification of an algebraic variety. This notion allows us to define the concept of a ball on a variety and to fix  a family of measures on it.

\subsection{The framework}
For a variety $\bfX$ defined over a finite field $\F_\ell$
we introduce the notion
of rectification (see \Cref{def:rect}). \RamiE{For an affine variety this is just an embedding \RamiK{in}to an affine space. For a general variety this is such an embedding for each chart in a finite atlas.} This notion allows us to define balls in the set $X:=\bfX(F)$ of  $F$-points of the variety, where $F$ is a local field containing $\F_\ell$  (see \Cref{def: balls and measures}).  
{\RamiK{T}he notion of a ball is defined simultaneously for all local fields of the type  $\F_{\ell^k}((t))$. This allows us to formulate uniform statements for all such fields. 
\RamiB{
\begin{notation}
    For a variety $\bfX$ and an integer $k\in \N$, we will consider two kinds of balls in 
    $\bfX(\F_{\ell^k}((t)))$.
    \begin{enumerate}
        \item \RamiG{Large}  balls, denoted by $B_m^{\bfX,k}$, see \Cref{def: balls and measures}\eqref{def: balls and measures:1}\eqref{def: balls and measures:a} for the formal definition. \RamiE{For affine varieties these are} balls around the origin  (\RamiE{in the ambient affine space}). Here $m\in \Z$ is the valuative radius of the ball, i.e. the actual radius is $\ell^{km}$. Usually, $m$ will be positive when considering such a ball.
    \item \RamiG{Small} balls, denoted by $B_m^{\bfX,k}(x)$, see \Cref{def: balls and measures}\eqref{def: balls and measures:1}\eqref{def: balls and measures:c} for the formal definition. \RamiE{For affine varieties, t}hese are balls of valuative radius $m\in \Z$ around $x\in \bfX(\F_{\ell^k}((t)))$. Here the integer $m$ is usually negative. 
    \end{enumerate}
\end{notation}
}
}

Although the balls themselves will depend on the rectification, all the statements that we will prove will not. This is due to the fact that for any two rectifications, one can compare between the corresponding balls. See \Cref{lem:indep-rect}\eqref{lem:indep-rect:1}.

Similarly, we will define the notion of a \RamiB{$\mu$-}rectification of smooth algebraic varieties (see \Cref{def:rect}). This notion allows us to fix measures on balls in $X$ (see \Cref{def: balls and measures}). Again, although the measures themselves will depend on the \RamiB{$\mu$-}rectification, the results that we will prove will not. This is established in \Cref{lem:indep-rect}\eqref{lem:indep-rect:2} and \Cref{lem:sm.mes.def}.
\RamiB{
\begin{remark}
    For the sake of simplicity, we work only with algebraic varieties defined over $\F_\ell$. This is enough for our purposes. However, we believe that with minor modifications, all the statements would be valid also for varieties defined over $\F_\ell[[t]]$ and, with slightly more modifications, also for varieties defined over $\F_\ell((t))$.    
\end{remark}
}
\subsection{Main results}
We prove the following results:
\begin{enumerate}
    \item Effective uniform continuity and boundedness of algebraic morphisms on balls. See \S\ref{sssec:eff.un.cont}.
    \item Effective version of the implicit function theorem. See  \S \ref{sssec:eff.imp}.
    \item The \RamiE{complement} of a small neighborhood around a closed subvariety $\bfZ\subset \bfX$ is controlled by large balls in the complement of $\bfZ$. See \S \ref{sssec:devis}.
    \item Effective surjectivity of Nisnevich covers. See \S  \ref{sssec:eff.sur}.
        \item Effective smoothness of push forward of smooth measures with respect to smooth maps. See \S\ref{sssec:eff.sm}.
    \item Effective bounds on pushforward of smooth measures with respect to smooth maps. See \S\ref{sssec:eff.push.sm.bnd}. 
\end{enumerate}

\subsubsection{Effective uniform continuity and boundedness} \label{sssec:eff.un.cont}

Let $\gamma:\bfX\to \bfY$ be a  map of  algebraic varieties defined over a finite field $\F_\ell$.  
\Rami{This gives maps
$\gamma_k:\bfX(\F_{\ell^k}((t)))\to \bfY(\F_{\ell^k}((t)))$.}
\Rami{Note that each map $\gamma_k$ is uniformly continuous and bounded on any ball in  $\bfX(\F_{\ell^k}((t)))$.} We prove that the modulus of continuity and the bound on $\gamma_k$ in a ball \RamiB{$B_m^{\bfX,k}$ of a fixed (valuative) radius $m$} are bounded  when we vary $k$. 

More formally, we prove:
\begin{introprop}[\Cref{lem:cont.bnd}]\label[introprop]{prop: Eff. uniformly continuous and bounded }
Let $\gamma:\bfX\to \bfY$ be a map of rectified algebraic varieties defined over  a finite field $\F_\ell$.  Then for any $m\in \N$ there is $m'>m$ such that for any $k\in \N$ we have: 
\begin{enumerate}[(i)]
    \item \RamiB{$\gamma(B_m^{\bfX,k})\subset B_{m'}^{\RamiE{\bfY},k}$.}
    \item For any \RamiB{$x\in B_m^{\bfX,k}$, we have $\gamma(B_{-m'}^{\bfX,k}(x))\subset B_{-m}^{\RamiE{\bfY},k}(\gamma(x))$}.
\end{enumerate}
\end{introprop}

\subsubsection{Effective versions of the \RamiB{inverse and the} implicit function theorems}\label{sssec:eff.imp}
We prove an effective versions  of the \RamiB{inverse and the}  implicit function theorems. Informally it means the following:

Let $\gamma:\bfX\to \bfY$ be an \RamiB{\RamiE{\'etale} (respectively smooth)} map of smooth algebraic varieties defined over  a finite field $\F_\ell$.  
{We again consider the maps
$\gamma_k:\bfX(\F_{\ell^k}((t)))\to \bfY(\F_{\ell^k}((t)))$. Then\RamiB{, in a ball $B_{m}^{\bfX,k}$ the map $\gamma_k$ admits a local \RamiB{inverse (respectively section)}, with bounded  modulus of continuity, when $m\in \N$ is fixed and $k$  varies.}}

More formally, we prove \RamiB{the following theorems.

\begin{introtheorem}[\Cref{lem:et.mono}]\label{thm:main.et.mono}    
Let $\gamma:\bfX\to \bfY$ be an \RamiB{\'etale} map of smooth rectified algebraic varieties defined over  a finite field $\F_\ell$.  Then for any $m$ there is $m'$ such that $\gamma|_{B_{m}^{\bfX,k}}$ is a monomorphism on \RamiH{small} balls of  valuative radius $-m'$.
\end{introtheorem}
}
\begin{introtheorem}[\Cref{lem:eff.imp}]\label{thm:main.imp}
Let $\gamma:\bfX\to \bfY$ be a smooth map of smooth rectified algebraic varieties defined over  a finite field $\F_\ell$.  Then for any $m$ there is $m'$ such that for any $k$ and any \RamiB{$x\in {B_m^{\bfX,k}}$ we have $$\gamma(B_{-m}^{\bfX,k}(x))\supset B_{-m'}^{\RamiE{\bfY},k}(\gamma(x)).$$}  
\end{introtheorem}

\subsubsection{Control on {the} \RamiE{complement} of a small neighborhood around a closed subvariety}\label{sssec:devis}
\begin{introprop}[\Cref{lem:dev}]\label[introprop]{prop: nbd.of.subvar}
 Let $\bfX$ be a rectified variety.
    Let $\bfU\subset \bfX$ be open and $\bfZ:=\bfX\smallsetminus \bfU$. Then for any $m$ there exists $m'$ such that for any $k$,  \RamiB{the ball $B_{m}^{\bfX,k}$} is covered by the union of \RamiB{the ball  $B_{m'}^{\bfU,k}$} and
    the \RamiE{union of balls of valuative radius $-m$ centered at points of $\bfZ(\F_{\ell^k}((t)))$}. 
    
\end{introprop}  
     Note that the notion of ball in $\bfU$ is not the one induced from $\bfX$, but rather an intrinsic notion for $\bfU$. In particular, $\bfZ$ is considered to be infinitely far with respect to this notion.

\subsubsection{Effectively surjective maps} \label{sssec:eff.sur}

We introduce the notion of effective surjectivity for a map $\gamma:\bfX\to\bfY$. Informally it means that the map is not only surjective on the level of points but we also can control the norm of a pre-image in terms of the norm of the target in a way that is uniform on extension of the local field. More precisely: 
\begin{definition}[\Cref{def:effsur}]
   Let $\gamma:\bfX\to\bfY$ be a map  between rectified algebraic varieties defined over $\F_\ell$. We say that   $\gamma$ is effectively surjective if for any $m$ there exist $m'$ such that for any $k$ \RamiB{we have $$\gamma(B_{m'}^{\bfX,k})
   \supset B_{m}^{\bfY,k}.$$}
\end{definition}
We prove the following criterion for this surjectivity:
\begin{introtheorem}[\Cref{prop:crit-effsurj}]\label{thm: crit. eff. surj.}
        Let $\gamma:\bfX\to \bfY$ be a smooth map of {rectified} algebraic varieties {defined over a finite field $\F_\ell$} that is  onto on the level of points for any field extension of $\F_\ell$. Then $\gamma$ is effectively surjective.
\end{introtheorem}
\Rami{
\begin{remark}
    In fact our argument also shows that if $\gamma$ is onto on the level of points for any infinite field extension of $\F_\ell$, then there is a finite extension of scalars of $\gamma$ which is effectively surjective.
\end{remark}
}

\subsubsection{Effective smoothness of push forward of smooth measures with respect to smooth map}\label{sssec:eff.sm}
Let $\gamma:\bfX\to \bfY$ be a smooth map of smooth algebraic varieties defined over  a finite field $\F_\ell$.  

Let $\mu$ and $\nu$ be compactly supported measures on $X=\bfX(\F_\ell((t)))$ and $Y=\bfY(\F_\ell((t)))$  which are coming from \RamiB{$\mu$-}rectifications  on $\bfX$ and $\bfY$.
Assume that the support of $\nu$ includes the support of $\gamma_*(\mu)$.
We show that the density of the pushforward $\gamma_*(\mu)$  with respect to $\nu$ is given by a smooth\footnote{\RamiE{i.e. locally constant.}} function which is constant on \RamiG{(small)} balls of some \RamiB{valuative} radius $-m$. Moreover $m$ remains bounded when we replace $\F_\ell$ with its finite extensions.

More formally, we make:
\begin{notation}[\Cref{def: balls and measures}]
For any integers $m,k$ and a \RamiB{$\mu$-rectified} smooth algebraic variety defined over $\F_\ell$ we use the rectification in order to define a measure  $\mu_m^{\bfX,k}$ on 
\Rami{$B_{\bfX,k}^m$}.
\end{notation}
We prove:
\begin{introtheorem}[\Cref{lem:push.mes.sm}]\label{thm: eff.smooth.of.push.of.measures }
    Let $\gamma:\bfX \to \bfY$ be a  smooth map  of  \RamiB{$\mu$-rectified}     
    varieties.      
    Then for any $m\in \N$ there \RamiH{are}  $\RamiH{m''>}m' \in \N$ such that for any $k\in \N$,
    \EitanH{any $0\leq m_0\leq m$}
    and any  function $g\in C^\infty_\RamiD{c}(\bfX(\F_{\ell^k}((t))))$ which is constant on balls of \RamiB{valuative radius $-m$,}  
    there is a function $f\in  C^\infty_\RamiD{c}(\bfY(\F_{\ell^k}((t))))$ 
     which is constant on balls of \RamiB{valuative radius $\RamiH{-m''}$,}  
    such that:
    $$\gamma_*(g\mu_{\RamiH{m_0}}^{\bfX,k})=f\cdot \mu_{m'}^{\bfY,k}.$$
\end{introtheorem}

\subsubsection{Effective bounds on push forward of smooth measures}\label{sssec:eff.push.sm.bnd}

We prove:
\begin{introtheorem}[\Cref{lem.sub.bnd}, \Cref{lem:bnd.on.sup}, \Cref{cor:push.bnd.above}]\label{thm: Eff. bounds on push of smooth measures}
    Let $\gamma:\bfX\to\bfY$ be a submersion of \RamiB{$\mu$-rectified} smooth varieties defined over $\F_\ell$. Then for any $m$ there are $m''>m'>m$ such that for any $k\in \N$  we have $$\ell^{-k\RamiE{m''}}\mu_{m}^{\bfY,k}\RamiE{1}_{\supp(\gamma_*(\mu_{m'}^{\bfX,k}))} \RamiE{\leq} \gamma_*(\mu_{m'}^{\bfX,k})\RamiE{\leq}\ell^{km''}\mu_{m''}^{\bfY,k}$$
    Moreover, if $\gamma$ is effectively surjective then 
    $$\ell^{-km'}\mu_{m}^{\bfY,k} \RamiE{\leq} \gamma_*(\mu_{m'}^{\bfX,k}).$$
\end{introtheorem}

\subsection{Related results}
Our notion of balls is parallel to the notion of norms in \cite[\S18]{Kot}. However, while \cite[\S18]{Kot} was concerned with large balls, we are also interested in small balls, and in measures. On the other hand, \cite[\S18]{Kot} \RamiE{provided} more quantitative results than \RamiE{provided here.}  

\Dima{The results of \cite{Kot} are, by themselves, not uniform on the field. Since \cite{Kot} allows non-local fields like $\bar \F_p((t))$, sometimes it is possible to deduce from it results that are uniform on field extensions. However, this is not the case for results that include the existential quantifier, like the implicit function theorem.}

\RamiE{It would be interesting to unify  the framework of this paper with the framework of  \cite{Kot} to get results  that are more quantitative, valid for any local field, and uniform on field extensions.}

Other related notions of metric in the context of Archimedean algebraic geometry was introduced in \cite{WE, MAG} under the name of metric algebraic geometry.

\subsection{Motivation}
The main  motivation for this work is a sequel work \cite{AGKS1} where we bound the dimensions of the jet schemes of the nilpotent cone (in $\fg\fl_n$)  in positive (small) characteristic. 

The characteristic zero counterpart of this result is done by \cite[Appendix by Eisenbud and Frenkel]{Mu}.
The methods of \cite{Mu} are not available in positive characteristic. Instead we use an analytic argument, resembling \cite{HC_VD}  in order to bound the number of points in these jet schemes. {These arguments involve volumes and integration.}

We use the Lang-Weil bounds to deduce the required bound on the dimension. For this we need the bound on the number of points to be uniform in field extensions of the underlying finite field. 
{Classical analytic arguments do not give such uniform bounds. Therefore they are not enough for us and }we need the results of the present paper.

In \cite{AGKS1}, we use the bound on the dimension of the jet-schemes in order to bound the push-forward of smooth measures under the Chevalley map $p:\fg\fl_n\to \fc_n$, that sends every matrix to its characteristic polynomial. We do it under the assumption of existence of a certain resolution of singularities.

Note that the $0$-characteristic counterpart  of \RamiK{the bound above} is done \cite{HC_VD}. While we can use the method of \cite{HC_VD} in a neighborhood of the nilpotent cone,
{the method fails to provide the desired global result due to issues of positive characteristic \RamiE{(the fact that the center does not split the Lie algebra; see more details in \cite[\S1.5.1]{AGKS1})}.} This is why we take the detour through the jet-schemes. 
\subsection{Ideas of the proofs}{A possible way to get uniform results for the fields $\F_{\ell^k}((t))$ is to work over the field $\bar \F_{\ell}((t))$. However, this field is not locally compact and we can not obtain boundedness in  the standard way. Thus, we have to use different methods, as we will now describe.}

The proof of \Cref{prop: Eff. uniformly continuous and bounded } is straightforward but technical.

\subsubsection{Idea of the proofs of  \Cref{thm:main.et.mono} and \Cref{thm:main.imp}}
\RamiB{The classical argument for the implicit function theorem}
(see e.g. \cite{Krantz})
is rather effective for the case of affine spaces, \RamiB{but not in general. So we have to employ additional considerations.}
Let us briefly explain the main steps in our \RamiE{proofs}:
\RamiE{
\begin{proof}[Idea of the proof of 
\Cref{thm:main.et.mono}]
Consider the ``universal" standard \'etale map. This is the projection map $\gamma_n: \bfX_n\to \bfY_n$ where $\bfY_n$ is the collection of monic polynomials of degree $n$ and
    $\bfX_n:=\{(f,a)|f \in \bfY_n; f(a)=0, f'(a)\neq 0\}$.
    
    We prove \Cref{thm:main.et.mono} for this map by a direct computation. From this we deduce the general case using standard arguments of algebraic geometry (see \cite[\href{https://stacks.math.columbia.edu/tag/02GH}{Section 02GH}]{SP}).
\end{proof}
}
\begin{proof}[\RamiE{Idea of the proof of  \Cref{thm:main.imp}}]
$ $
\begin{enumerate}[Step 1.]
    \item We deduce the case when the source $\bfX$ is a general variety and the target $\bfY$ is an affine space. This we do by noticing that locally $\bfX$ can be written as a fiber of a smooth {map} between affine spaces $\bfV\to \bfW$ and then applying the statement for the map $\bfV\to \bfW\times \bfY$.
    \item The general case: we use the fact that locally $\bfY$ can be mapped by an \'etale map to an affine space $\bfL$. Then we use the result for the composition $\bfX\to\bfY\to\bfL$. In order to deduce the required statement, we also  \RamiB{use \Cref{thm:main.et.mono}.}    
\end{enumerate}
\end{proof}
\subsubsection{Idea of the proof of \Cref{prop: nbd.of.subvar}}
We use a stratification argument in order to  reduce the statement to the case when $\bfZ$ is smooth. This case we deduce from \Cref{thm:main.imp} using the fact that $\bfZ$ is locally a fiber of a smooth map.
\subsubsection{Idea of the proof of \Cref{thm: crit. eff. surj.}}
\begin{enumerate}[Step 1.]
    \item We show (using Noetherian induction) that if a map is surjective on the level of points over every field, then the target admits a stratification such that the map admits a section for each strata. See \Cref{lem:onto-start} below.
    \item We note that if a map admits a section then it is effectively surjective.
    \item We show, using the effective version of the implicit function theorem (\Cref{thm:main.imp}), that if a smooth map is effectively surjective over a closed subvariety, then  it is effectively surjective over a controllable neighborhood of this subvariety.
    \item \RamiB{We use the previous step and Proposition  \ref{prop: nbd.of.subvar} to show that
    any smooth map that is effectively surjective over a closed subvariety and over its complement is effectively surjective.} 
    \item We deduce the theorem.
\end{enumerate}
\subsubsection{Idea of the proof of \Cref{thm: eff.smooth.of.push.of.measures }}
Here we use a standard technique of \RamiB{presentation} of a smooth map \RamiB{as a composition of simpler maps,} and the effective version of the implicit function theorem, \Cref{thm:main.imp}.
\subsubsection{Idea of the proof of \Cref{thm: Eff. bounds on push of smooth measures}}
The main part is the first inequality. 
Here the main difficulty is that we can not deduce the statement for a composition of morphisms from the statement for each morphism. However, we still can reduce to the case when the morphism is a composition of an \'etale map and a projection from a product with affine space. For such compositions we can explicitly compute the density function of the push forward. Then we can get the required bound using \Cref{thm:main.imp}.
\subsection{Structure of the paper}
In \S\ref{sec:prel} we fix the conventions that we use in this paper.

In \S\ref{ssec:ballsMe} we define the notion of rectified algebraic varieties, the notion of a ball in such  varieties and  define the measures $\mu_m^{\bfX,k}$ described above. We also show that these objects essentially do not depend on the rectification.
We also prove basic properties of these objects, including \Cref{prop: Eff. uniformly continuous and bounded }, and the second (easier) inequality of \Cref{thm: Eff. bounds on push of smooth measures}.

In \S\ref{ssec:imp} we prove effective version of the implicit function theorem (\Cref{thm:main.imp}) and draw some corollary of it (\Cref{prop: nbd.of.subvar} and the main part of \Cref{thm: Eff. bounds on push of smooth measures}).

In \S\ref{ssec:smoothfunme} we prove \Cref{thm: eff.smooth.of.push.of.measures }.

In \S\ref{sec:sur} we prove \Cref{thm: crit. eff. surj.} and complete the proof of \Cref{thm: Eff. bounds on push of smooth measures}.
\subsection{Acknowledgments}
We would like to thank Gal Binyamini and Yosef Yomdin for helpful discussions. \DimaA{We also thank the  referee for finding a mistake in an earlier version of the paper. See \Cref{rem:ref}}.  

During the preparation of this paper, A.A., D.G. and E.S. were partially supported by the ISF grant no. 1781/23. D.K. was partially supported by \RamiC{the ERC grant no. 101142781}.

\section{Conventions}\label{sec:prel}
\begin{enumerate}
    \item Throughout we     
    \RamiB{fix a prime power $\ell$.}
    \item \RamiE{We normalize the absolute value on $\F_{\ell^\RamiJ{k}}((t))$ by $|t|=\frac 1{\ell^\RamiJ{k}}$.}
        \item By a variety we mean a reduced \RamiE{separated} scheme of finite type over a field.
    Unless stated otherwise this field will be $\F_\ell$.\index{variety}
        \item \RamiB{Morphisms between varieties will be always defined over the field of definition of the varieties.}
        \item When we consider a fiber product of varieties, \Rami{and fibers of maps between varieties,} we always consider it in the category of schemes.

    \item We will describe \Rami{subschemes} and morphisms of varieties \Rami{and schemes}  using set-theoretical language, when no ambiguity is possible.    
    \item We will usually denote algebraic varieties by bold face letters (such as $\bfX$).    
    \item \RamiH{For an algebraic variety $\bfX$ and a regular function  $f\in O(\bfX)$ we denote by $\bfX_f\subset \bfX$ the principal open set attached to $f$.}
    \item We will use the symbol $\square$\index{$\square$} in a middle of a square diagram in order to indicate that the square is Cartesian. 
    \item \Rami{For a field extension $E/F$ and a variety $\bfX$ defined over $F$, we will denote by $\bfX_E$ the extension of scalars. We will use analogous notation for all algebro-geometric objects on the variety $\bfX$, including regular functions and differential forms.}
    \item  \Rami{For a regular function $f$ on an algebraic variety $\bfX$ defined over a local field $F$,  we denote by $|f|$ 
    the corresponding function on $X:=\bfX(F)$.}\index{$\vert f \vert$}
    \RamiE{That is the function given by $x\mapsto |f(x)|$.}
    \item  \Rami{For a top form $\omega$ on a smooth algebraic variety $\bfX$ defined over a local field $F$,  we denote the corresponding measure on $X:=\bfX(F)$ by $|\omega|$.}\index{$\vert\omega\vert$}
\end{enumerate}


\section{Balls and measures on rectified varieties}\label{ssec:ballsMe}
In this section we introduce the concept of rectified varieties and define balls on them. We also fix a measure on each ball.

We prove that the balls and measures are  essentially independent of the rectification, see \Cref{lem:indep-rect}.

We also prove some basic properties of these objects, including \Cref{prop: Eff. uniformly continuous and bounded } and the second (easier) inequality of \Cref{thm: Eff. bounds on push of smooth measures}.

\begin{defn}\label[defn]{def:rect}
    Let $\bfX$ be a smooth algebraic variety over $\F_\ell$.  
    
    \begin{enumerate}
        \item 
    A rectification\index{rectification, rectified variety}   of $\bfX$ is a finite open cover $\bfX \subset \bigcup_{\alpha \in I} \bfU_{\alpha}$ 
    with 
    closed embeddings $i_\alpha:\bfU_{\alpha}\to \bA^M$ \RamiD{for each $\alpha \in I$.}    
    \item We will call a rectification simple if $|I|=1$.
    \item     By a rectified variety we will mean a smooth algebraic variety over $\F_\ell$ equipped with a rectification. \RamiH{We say that the variety is simply rectified if this rectification is simple.} By a map or a morphism of such we just mean a morphism of the underlying algebraic varieties.
    \item     A \RamiB{$\mu$-}rectification of  $\bfX$  is a rectification of $\bfX$  together with invertible top differential forms $\omega_{\alpha}\in \Omega^{top}(\bfU_{\alpha})$ for each $\alpha \in I$.
    \item We define similarly the notion of a \RamiB{$\mu$-rectified} variety, and 
     simple \RamiB{$\mu$-rectification}.
    \end{enumerate}
\end{defn}

\begin{definition}
    By an almost affine space we mean a principal open subset\footnote{\RamiB{i.e.  the complement of a divisor}} in an affine space defined over $\F_\ell$. Note that  any almost affine space is equipped with a \Rami{natural} simple \RamiB{($\mu$-)}rectification \RamiA{that we will call the standard \RamiB{($\mu$-)}rectification on this space}.
\end{definition}

\RamiE{
\begin{definition}\label[definition]{def:simp balls}
$ $
    Let $(\bfX,i)$ be a simply rectified variety. Then, for any $k\in \N$ and $m\in \Z$ define:  \begin{enumerate}
    \item \label{def:simp balls:1}  
$B_m^{\bfX,k}:=i^{-1} \left(t^{-m}\F_{\ell^k}[[t]]^{M}\right)\RamiB{\subset \bfX(\F_{\ell^k}((t)))}.$ \index{$B_m^{\bfX,k}$}
        \item\label{def:simp balls:2} For $x\in \bfX(\F_{\ell^{k}}((t)))$ define {a ball around $x$:}
        $$B_m^{\bfX,k}(x):= i^{-1} \left(i(x)+t^{-m}\F_{\ell^k}[[t]]^{M}\right).$$ \index{$B_m^{\bfX,k}(x)$}
\end{enumerate}
\end{definition}
}
\RamiE{

The next Lemma follows immediately from Hilbert's Nullstellensatz.
}
\RamiE{

\begin{lem}\label{lem:trans}
    Let $(\bfX,\bfU_\alpha,i_\alpha)_{\alpha\in I}$ be a rectified variety and let $\alpha,\beta\in I$. Then $\bfU_\alpha\cap \bfU_\beta$ is affine, the map $(i_\alpha,i_\beta)$ is a closed embedding of $\bfU_\alpha\cap \bfU_\beta$ into $\A^{2M}$, and there exist:
    \begin{enumerate}
        \item\label{lem:trans:1} polynomials $f_1,\dots,f_q$ in the coordinates of $i_\alpha$, with coefficients in $\F_\ell$, such that $\bfU_\alpha\cap \bfU_\beta=\bigcup_{i=1}^{q} (\bfU_\alpha)_{f_i}$;
        \item\label{lem:trans:2} an integer $d\geq 1$ and polynomials $g_{ij}$ in the coordinates of $i_\alpha$, with coefficients in $\F_\ell$, such that $i_{\beta,j}=g_{ij}/f_i^{d}$ on $(\bfU_\alpha)_{f_i}$ for all $i$ and all $1\leq j\leq M$;
        \item\label{lem:trans:3} polynomials $h_1,\dots,h_q$ in the coordinates of $i_\alpha$ and of $i_\beta$ together, with coefficients in $\F_\ell$, such that $\sum_{i=1}^{q} f_ih_i=1$ on $\bfU_\alpha\cap \bfU_\beta$.
    \end{enumerate}
\end{lem}
}
\RamiE{

\begin{lem}\label{lem:KL}
Let $(\bfX,\bfU_\alpha,i_\alpha)_{\alpha\in I}$ be a rectified variety. Then there
is $s\in\N$ with the following property. For every $k\in\N$, all $\alpha,\beta\in I$,
every $m\in\N_{>0}$, every
\[
x\in B^{\bfU_\alpha,k}_{m}\cap B^{\bfU_\beta,k}_{m},
\]
every integer $r\ge1$ 
we have 
$$B_{-r-sm}^{\bfU_\alpha,k}(x)\subset  B_{-r}^{\bfU_\beta,k}(x)$$
\end{lem}
}
\RamiE{

\begin{proof}
    For $M\in \N$ and  $v\in \Fk^M$ write $|v|:=\max_{j=1,\cdots,M} |v_j|$.

    For any pair $\alpha,\beta$, 
    let $f_i,h_i,d$ and $g_{ij}$ be as in \Cref{lem:trans}.    
    Let $D$ be the maximum of the degrees all these polynomials when $(\alpha,\beta) \in I \times I$  ranges over all pairs in our cover. Also \RamiE{choose} $d$ s.t. it will fit all pairs $(\alpha,\beta)$.
    Set $$s:=D(d+3).$$ 
    
    Fix now $k$, $\alpha$, $\beta$, $m$, $x$ and $r$ as in the statement, let $y\in B_{-r-sm}^{\bfU_\alpha,k}(x)$, and put $A:=|t|^{-m}$. Then $A\geq 1$ and 
    \begin{equation}\label{eq:abs.<.A}
        |i_\alpha(x)|\leq A,\qquad |i_\beta(x)|\leq A,\qquad |i_\alpha(y)-i_\alpha(x)|\leq |t|^{r+sm}=|t|^{r}A^{-s},
        \end{equation}
    and therefore also $|i_\alpha(y)|\leq A$. 

    \begin{enumerate}[{Step} 1.]
        \item  $y\in \bfU_\beta(\Fk)$.\\    
    Take the data of \Cref{lem:trans} for the ordered pair $(\alpha,\beta)$. Evaluating $\sum_{i=1}^{q} f_ih_i=1$ at $x$ gives $$\sum_{i=1}^{q} f_i\bigl(i_\alpha(x)\bigr)\,h_i\bigl(i_\alpha(x),i_\beta(x)\bigr)=1.$$ \RamiH{Using} that \RamiH{the degrees of $h_i$ are at most $D$ and \eqref{eq:abs.<.A}} we get $$1\leq \max_i \bigl|f_i\bigl(i_\alpha(x)\bigr)\bigr|\cdot A^{D},$$ so that $\bigl|f_{i_0}\bigl(i_\alpha(x)\bigr)\bigr|\geq A^{-D}$ for some $i_0$. On the other hand, since $r\geq 1$, $s\geq 2D-1$ and $A\geq 1$, $$\bigl|f_{i_0}\bigl(i_\alpha(y)\bigr)-f_{i_0}\bigl(i_\alpha(x)\bigr)\bigr|\leq |t|^{r}A^{-s}A^{D-1}\leq |t|\,A^{D-1-s}\leq |t|\,A^{-D}<A^{-D}\leq \bigl|f_{i_0}\bigl(i_\alpha(x)\bigr)\bigr|.$$ Hence $\bigl|f_{i_0}\bigl(i_\alpha(y)\bigr)\bigr|=\bigl|f_{i_0}\bigl(i_\alpha(x)\bigr)\bigr|\geq A^{-D}$. In particular, $f_{i_0}\bigl(i_\alpha(y)\bigr)\neq 0$, so $$y\in (\bfU_\alpha)_{f_{i_0}}\subset \bfU_\alpha\cap \bfU_\beta,$$ and in particular $y\in \bfU_\beta(\Fk)$.
    
            \item  $y\in B_{-r}^{\bfU_\beta,k}(x)$.\\
    Write $f:=f_{i_0}$, $g_j:=g_{i_0j}$. By the previous step $\phi:=\bigl|f\bigl(i_\alpha(x)\bigr)\bigr|=\bigl|f\bigl(i_\alpha(y)\bigr)\bigr|$, so that $A^{-D}\leq \phi\leq A^{D}$. For each $1\leq j\leq M$ we have, by \Cref{lem:trans}\eqref{lem:trans:2},
    $$i_{\beta,j}(y)-i_{\beta,j}(x)=\frac{\Bigl(g_j\bigl(i_\alpha(y)\bigr)-g_j\bigl(i_\alpha(x)\bigr)\Bigr)f\bigl(i_\alpha(x)\bigr)^{d}+g_j\bigl(i_\alpha(x)\bigr)\Bigl(f\bigl(i_\alpha(x)\bigr)^{d}-f\bigl(i_\alpha(y)\bigr)^{d}\Bigr)}{\Bigl(f\bigl(i_\alpha(x)\bigr)f\bigl(i_\alpha(y)\bigr)\Bigr)^{d}}.$$
    Now $\bigl|g_j\bigl(i_\alpha(y)\bigr)-g_j\bigl(i_\alpha(x)\bigr)\bigr|\leq |t|^{r}A^{-s}A^{D-1}$, $\bigl|g_j\bigl(i_\alpha(x)\bigr)\bigr|\leq A^{D}$ and $$\bigl|f\bigl(i_\alpha(x)\bigr)^{d}-f\bigl(i_\alpha(y)\bigr)^{d}\bigr|\leq \bigl|f\bigl(i_\alpha(x)\bigr)-f\bigl(i_\alpha(y)\bigr)\bigr|\cdot \phi^{d-1}\leq |t|^{r}A^{-s}A^{D-1}\phi^{d-1},$$ so that the numerator is bounded by $$|t|^{r}A^{D-1-s}\phi^{d-1}\max(\phi,A^{D})=|t|^{r}A^{2D-1-s}\phi^{d-1}.$$ Dividing by $\phi^{2d}$ and using $\phi^{-d-1}\leq A^{D(d+1)}$ we obtain $$|i_\beta(y)-i_\beta(x)|\leq |t|^{r}A^{2D-1+D(d+1)-s}\leq|t|^{r},$$ that is, $y\in B_{-r}^{\bfU_\beta,k}(x)$.
\end{enumerate}
    
\end{proof}
}
\RamiE{
We choose for any rectified variety an $s$ as in the above lemma and denote it by $s_\bfX$\index{$s_\bfX$}. We choose it in a way that if the rectification is simple then $s=0$.

We are now ready to define balls and measures on our rectified varieties.
}
\RamiE{

\begin{definition}\label[definition]{def: balls and measures}
$ $
\begin{enumerate}
    \item \label{def: balls and measures:1}
    Let $(\bfX,\bfU_\alpha,i_\alpha)$ be a rectified variety. Then, for any $k\in \N$ and $m\in \Z$ define:    
            \begin{enumerate}
            \item\label{def: balls and measures:a}
$B_m^{\bfX,k}:=\bigcup_{\alpha} B_m^{\bfU_\alpha,k}{\subset \bfX(\F_{\ell^k}((t)))}.$\index{$B_m^{\bfX,k}$}
            \item
$B_\infty^{\bfX,k}:=\bigcup_{m\in \N} B_m^{\bfX,k}=\bfX(\F_{\ell^k}((t)))$. 
\RamiK{We will also use this notation for general algebraic varieties (not necessarily rectified or smooth).}\index{$B_\infty^{\bfX,k}$}
        \item\label{def: balls and measures:c} For $x\in \bfX(\F_{\ell^{k}}((t)))$ define {a ball around $x$:}
        $$B_m^{\bfX,k}(x):= \bigcup_{\alpha, a \text{ s.t. } x\in B_{a}^{\bfU_\alpha,k}; a>0} B_{m-as_\bfX  }^{\bfU_\alpha,k}(x) .$$ \index{$B_m^{\bfX,k}(x)$}
        \item    For $\bfZ\subset \bfX$  we define:  $$B_m^{\bfX,k}(\bfZ):=\bigcup_{z\in \RamiD{B^{\bfZ,k}_\infty}}   B_m^{\bfX,k}(z).$$\index{$B_m^{\bfX,k}(\bfZ)$} 
    \end{enumerate}
        
        \item 
    Let $(\bfX,\bfU_\alpha,i_\alpha,\omega_\alpha)$ be a \RamiB{$\mu$-rectified} variety. Then, for any \RamiD{integers $k\in\N, m\in \Z$}  we define
        a measure $\mu_m^{\bfX,k}$ on $\bfX(\F_{\ell^k}((t)))$ supported on $B_m^{\bfX,k}$ by: $$\mu_m^{\bfX,k}:=\sum_\alpha |(\omega_\alpha)_{\F_{\ell^k}((t))}| \cdot 1_{i_\alpha^{-1} \left(t^{-m}\F_{\ell^k}[[t]]^{M}\right)}.$$ \index{$\mu_m^{\bfX,k}$}
    \item If $\bfX$ is an affine space, we denote by $\mu^{\bfX,k}$ the Haar measure on $\bfX(\F_{\ell^k}((t)))$  normalized such that $\mu^{\bfX,k}(\bfX(\F_{\ell^k}[[t]])=1$.\index{$\mu^{\bfX,k}$}
\end{enumerate}
\end{definition}
}
\RamiE{
\begin{remark}\label{rem:ref}
    In \Cref{def: balls and measures}\eqref{def: balls and measures:1}\eqref{def: balls and measures:c} it is important that we do not take the union of all balls in the chart of the same radius $m$, but rather use the constant $s_\bfX$ to decrease the radius. Otherwise the resulting ball will be too large, invalidating most of the results of this paper that concern \RamiG{small} balls. 
    
    Indeed, if $\bfX=\bA^2$ and one of the charts is $\bfU_1=(\bA^1\smallsetminus 0) \times \bA^1$ then, with an appropriate embedding of $\bfU_1$ to $\bA^3$, for any $m$ one can find a point with arbitrarily large $y$ coordinate that lies in $B_{m}^{\bfU_1,k}(x)$ for a point $x$ on the $x$ axis.
\end{remark}
}

\begin{lem}\label{lem:tri}
    Let $\bfX$ be a rectified variety. Then for any 2 positive integers $m_1,m_2\in \N$ we have:
    \begin{enumerate}
        \item\label{lem:tri:mon} If $x\in B^{\bfX,k}_{m_2}$ then $B_{-m_1}^{\bfX,k}(x)\sub B^{\bfX,k}_{m_2}$.
        \item\label{lem:tri:sym} If $x\in B^{\bfX,k}_{\infty}$ and $y\in B^{\bfX,k}_{-m_1}(x)$ then $x\in B^{\bfX,k}_{-m_1}(y)$.
        \item\label{lem:tri:ultra} If the rectification of $\bfX$ is simple and  $m_1\geq m_2$, then
        $$B_{-m_1}^{\bfX,k}(x)\sub B_{-m_2}^{\bfX,k}(x)=B^{\bfX,k}_{-m_2}(y)$$ for any $x\in \bfX(\F_{\ell^k}((t)))$ and $y\in  B_{-m_2}^{\bfX,k}(x)$.
    \end{enumerate}
\end{lem}
\RamiE{
\begin{proof}
    Write $s:=s_\bfX$. Recall that the index $a$ in \Cref{def: balls and measures}\eqref{def: balls and measures:c} ranges over $\N$, and that for any $\alpha$ and any $z\in \bfU_\alpha(\Fk)$ we have $z\in B_a^{\bfU_\alpha,k}$ if and only if $\val(i_\alpha(z))\geq -a$.
    \begin{enumerate}
        \item Let $\beta$ be such that $x\in B_{m_2}^{\bfU_\beta,k}$, let $y\in B_{-m_1}^{\bfX,k}(x)$, and let $\alpha$ and $a\in \N$ be such that $x\in B_{a}^{\bfU_\alpha,k}$ and $y\in B_{-m_1-as}^{\bfU_\alpha,k}(x)$.

        If $a\leq m_2$ then, since $m_1+as\geq 0$,
        $$\val(i_\alpha(y))\geq \min\bigl(\val(i_\alpha(x)),m_1+as\bigr)\geq -a\geq -m_2,$$
        so $y\in B_{m_2}^{\bfU_\alpha,k}\sub B_{m_2}^{\bfX,k}$.

        If $a> m_2$ then $x\in B_{a}^{\bfU_\alpha,k}\cap B_{a}^{\bfU_\beta,k}$, so \Cref{lem:KL}, applied with $m=a$ and $r=m_1$, gives
        $$y\in B_{-m_1-as}^{\bfU_\alpha,k}(x)\sub B_{-m_1}^{\bfU_\beta,k}(x).$$
        Hence $\val(i_\beta(y))\geq \min\bigl(\val(i_\beta(x)),m_1\bigr)\geq -m_2$, and again $y\in B_{m_2}^{\bfU_\beta,k}\sub B_{m_2}^{\bfX,k}$.
        \item Let $\alpha$ and $a\in \N$ be such that $x\in B_{a}^{\bfU_\alpha,k}$ and $y\in B_{-m_1-as}^{\bfU_\alpha,k}(x)$. Since $m_1+as\geq 0$ we get,  $\val(i_\alpha(y))\geq -a$. \RamiH{Hence}  $y\in B_{a}^{\bfU_\alpha,k}$. \RamiH{This implies:
        $$B_{-m_1-as}^{\bfU_\alpha,k}(y)\sub B_{-m_1}^{\bfX,k}(y).$$
        Thus
        }
        $$x\in B_{-m_1-as}^{\bfU_\alpha,k}(y)\sub B_{-m_1}^{\bfX,k}(y).$$
        \item In this case $s=0$ and there is only one chart, so $B_{-m}^{\bfX,k}(x)$ is the ball of \Cref{def:simp balls}\eqref{def:simp balls:2} for any $m\in \N$ and any $x$. The assertion follows from the inclusion $t^{m_1}\Ok\sub t^{m_2}\Ok$ and from the equality $i(x)+t^{m_2}\Ok^M=i(y)+t^{m_2}\Ok^M$, which holds whenever $i(x)-i(y)\in t^{m_2}\Ok^M$.
    \end{enumerate}
\end{proof}
}
\RamiE{
\begin{prop}\label{prop:price}
Let $(\bfX,\bfU_\alpha,i_\alpha)$ be a rectified variety.
Let $\alpha$ be an index, let $m\in\N_{>0}$, let $x\in B^{\bfU_\alpha,k}_{m}$ and let
$r\ge1$. Then
\begin{enumerate}
\item\label{price:up} 
$B^{\bfU_\alpha,k}_{-r-ms_{\bfX}}(x)\subseteq B^{\bfX,k}_{-r}(x)$;
\item\label{price:down}
$B^{\bfX,k}_{-r-ms_{\bfX}}(x)\subseteq B^{\bfU_\alpha,k}_{-r}(x)$.
\end{enumerate}
\end{prop}
}
\RamiE{
\begin{proof}
\eqref{price:up} follows from the definition.

For \eqref{price:down}, let $y\in B^{\bfX,k}_{-r-ms_{\bfX}}(x)$. So we have a pair $(\beta,b)$, with $b\in \N$, such that $x\in B^{\bfU_\beta,k}_{b}$ and
$y\in B^{\bfU_\beta,k}_{-r-ms_{\bfX}-bs_{\bfX}}(x)$. Put $m_0:=\max(b,m)$. Then
$x\in B^{\bfU_\beta,k}_{m_0}\cap B^{\bfU_\alpha,k}_{m_0}$, and
$r+ms_{\bfX}+bs_{\bfX}\ge r+s_{\bfX}m_0$, so
\[
y\in B^{\bfU_\beta,k}_{-r-s_{\bfX}m_0}(x)\subseteq B^{\bfU_\alpha,k}_{-r}(x)
\]
by  \Cref{lem:KL}. 
\end{proof}
}
\RamiE{
\begin{cor}\label{cor:price.Z}
Let $(\bfX,\bfU_\alpha,i_\alpha)$ be a rectified variety. Let $\bfZ\subset \bfX$ be a closed subvariety.
Let $\alpha$ be an index, let $m\in\N_{>0}$, let
$r\ge1$. Then
$$B^{\bfU_\alpha,k}_{-r-ms_{\bfX}}(\bfZ \cap \bfU_\alpha) \cap B^{\bfU_\alpha,k}_{m}\subseteq B^{\bfX,k}_{-r}(\bfZ).$$ 
\end{cor}
}
\RamiE{
\Cref{prop:price} also gives
\begin{cor}\label{cor:price.no.af}
Let $\bfX=\bigcup \bfV_i$ be a finite open cover of a smooth algebraic variety. Equip \RamiH{each} $\bfV_i$ with a rectification and equip $\bfX$ with the rectification obtained from these.

Then there exists $s$ s.t. for any $m\in \N_{>0}$ and any $r\geq 1$ we have:
        \begin{enumerate}        
            \item\label{price.no.af:down}   For any  
        $x\in B_{m}^{\bfX,k}$
        we have 
        $$B_{-r-ms}^{\bfX,k}(x)\subset        
        \bigcup_{i \text{ s.t. } x\in B_{m}^{\bfV_i,k}} 
        B_{-r}^{\bfV_i,k}(x)$$
            \item\label{price.no.af:up} 
             For any  $i$ and any $x\in B_{m}^{\bfV_i,k}$ we have 
            $$B_{-r}^{\bfX,k}(x)\supset        
        B_{-r-ms}^{\bfV_i,k}(x)$$
        \end{enumerate}
        
\end{cor}
}
\RamiE{
\begin{proof}   
Let $\bfV_i=\bigcup\bfV_{ij}$ be the union given by the rectification. 
Take $s= s_\bfX+\max s_{\bfV_i}.$
\begin{enumerate}
    \item 
    For any 
        $x\in B_{m}^{\bfX,k}$ we have 
\begin{align*}
B_{-r-ms}^{\bfX,k}(x)&\overset{\text{Prop }\ref{prop:price}\eqref{price:down}}{\subset} \bigcup_{i,j \text{ s.t. } x\in B^{\bfV_{ij},k}_m} B^{\bfV_{ij},k}_{-r-ms+m s_\bfX}(x)\\
        &\overset{\text{Prop }\ref{prop:price}\eqref{price:up}}{\subset} 
        \bigcup_{i,j \text{ s.t. } x\in B^{\bfV_{ij},k}_m} B^{\bfV_{i},k}_{-r-ms+m s_\bfX+ m s_{\bfV_i}}(x)\\
        &\subset
        \bigcup_{i \text{ s.t. } x\in B_{m}^{\bfV_i,k}} 
        B_{-r}^{\bfV_i,k}(x)
\end{align*}
        \item 
        For any  $i$ and any $x\in B_{m}^{\bfV_i,k}$ we have 
\begin{align*}
B_{-r}^{\bfX,k}(x)&\overset{\text{Prop }\ref{prop:price}\eqref{price:up}}{\supset} 
\bigcup_{j \text{ s.t. } x\in B^{\bfV_{ij},k}_m}  B^{\bfV_{ij},k}_{-r-m s_\bfX}(x)\\
        &\overset{\text{Prop }\ref{prop:price}\eqref{price:down}}{\supset} 
        \bigcup_{j \text{ s.t. } x\in B^{\bfV_{ij},k}_m} B^{\bfV_{i},k}_{-r-m s_\bfX- m s_{\bfV_i}}(x)\\
        &\supset
        B_{-r-sm}^{\bfV_i,k}(x)
\end{align*}

\end{enumerate}
\end{proof}
        }

\begin{prop}\label[prop]{lem:cont.bnd}
Let $\gamma:\bfX\to \bfY$  be map of rectified algebraic varieties.  Then 
for any $m\in \N$ there is $m'>m$ such that for any $k$ and any $x\in B^{\bfX,k}_m$ we have 
\begin{enumerate}[(i)]
    \item \label{lem:cont.bnd:S1}$\gamma(B^{\bfX,k}_{m})\subset B^{\bfY,k}_{m'}$.
    \item \label{lem:cont.bnd:S2}$\gamma(B^{\bfX,k}_{-m'}(x))\subset B^{\bfY,k}_{-m}(\gamma(x))$.
\end{enumerate}
\end{prop}
\begin{proof}
$ $
\RamiE{In the reductions below we pass repeatedly between a ball of a rectified variety and a ball of one of its charts. Since the centre is always bounded by the fixed $m$, this is done by \Cref{prop:price}, at the cost of enlarging the constants.}
\begin{enumerate}[{Case} 1.]
    \item $\bfX=\A^M, \bfY=\A^1$ (both equipped with the standard rectifications) and $\gamma$ is a monomial:\\
    In this case it is easy to see that one can take $m'=dm$ where $d$ is the degree of $\gamma$.
    \item $\bfX=\A^M, \bfY=\A^1$ (both equipped with the standard rectifications):\\
    follows from the previous case and the fact that $\F_{\ell^k}((t))$ is non Archimedean.
    \item $\bfX, \bfY$ are affine spaces (both equipped with the standard rectifications):\\
    follows from the previous case.
    \item\label{lem:cont.bnd:4} The rectifications of $\bfX, \bfY$ are simple:\\
    by definition of a  map between affine varieties we have a commutative diagram 
    $$
    \begin{tikzcd}
     \bfX\arrow["",r]\arrow["\gamma",d]& \A^{M_1} \arrow["\gamma'",d] \arrow["",d]\\
    \bfY\arrow["",r]   & \A^{M_2}    
\end{tikzcd}
$$
where 
the rectifications of $\bfX$ and $\bfY$  are induced from their embedding into affine spaces in the diagram.
Now, this case follows from the previous one.
    \item\label{lem:cont.bnd:5} The rectification of $\bfY$ is simple:\\
    Follows from the previous case. \RamiE{In more detail, \eqref{lem:cont.bnd:S1} holds since $B_m^{\bfX,k}=\bigcup_\alpha B_m^{\bfU_\alpha,k}$ and the cover is finite. For \eqref{lem:cont.bnd:S2}, let $m''$ be given by the previous case for all the restrictions $\gamma|_{\bfU_\alpha}$, and take $m'=m''+ms_\bfX$. Given $x\in B_m^{\bfX,k}$, choose $\alpha$ with $x\in B_m^{\bfU_\alpha,k}$; by \Cref{prop:price}\eqref{price:down} we have $B_{-m'}^{\bfX,k}(x)\subset B_{-m''}^{\bfU_\alpha,k}(x)$, and the previous case applies to $\gamma|_{\bfU_\alpha}$.}
    \item The rectification of $\bfX$ is simple, $\gamma$ is an isomorphism, and the rectification of $\bfY$ comes from a cover with principal open sets, and their standard embedding into $\bfX\times \A^1$:\\
    Let $\bfX=\bfY=\bigcup_{i=1}^{M} \bfY_{f_i}=\bigcup \bfX_{f_i}$ be the cover defining the rectification of $\bfY$. By 
    Hilbert's Nullstellensatz
    we can find $g_i\in O_\bfX(\bfX)$ such that 
    \begin{equation}\label{eq:ser}
     \sum f_ig_i=1.       \end{equation}     
    Consider the map $\bfX\to \A^{2M}$ given by $x\mapsto (f_1(x),\dots,f_M(x),g_1(x),\dots,g_M(x))$ and apply to it Case \ref{lem:cont.bnd:4}. We obtain  numbers $m_1>m_0>m$ such that for any $x\in B_m^{\bfX,k}$ we have 
    \begin{enumerate}
        \item $|g_i(x)|<|t^{-m_0}|$\label{lem:cont.bnd:5:a}
        \item $f_i(B_{-m_{\RamiJ{1}}}^{\bfX,k}(x)) \subset B_{-m-2m_0\RamiE{-m_0s_\bfY}}^{\A^1,k}(f_{i}(x))$\label{lem:cont.bnd:5:b}
    \end{enumerate}
    \RamiH{Enlarge $m_1$ if necessary so that $m_1\geq m+m_0s_\bfY$}.
    Take $m'=m_1$.
    From \eqref{lem:cont.bnd:5:a}
    \RamiH{and \eqref{eq:ser}} we obtain that for any $x\in B_m^{\bfX,k}$ there exists $i$ such that   
    \begin{equation}\label{fi>t}
     |f_i(x)|> |t^{m_0}|.   
    \end{equation}
    Thus $$x\in B_{m_0}^{\bfY_{f_i},k}\subset B_{m_0}^{\bfY,k}\subset B_{m'}^{\bfY,k}$$ proving \eqref{lem:cont.bnd:S1}.

    From \eqref{fi>t} and \eqref{lem:cont.bnd:5:b} we obtain that for any $x\in B_m^{\bfX,k}$ there exists $i$ such that for any $y\in B_{-m'}^{\bfX,k}(x)$ we have 
    $$\left|\frac{1}{f_i(x)}-\frac{1}{f_i(y)}\right|=\left|\frac{f_i(y)-f_i(x)}{f_i(x)f_i(y)}\right|<|f_i(y)-f_i(x)|\cdot  |t^{-2m_0}|\leq |t^{m\RamiE{+m_0s_\bfY}}|.$$
    Thus
    $$y \in B_{-m\RamiE{-m_0s_\bfY}}^{\bfY_{f_i},k}(x)\subset B_{-m}^{\bfY,k}(x),$$
    \RamiE{the last inclusion by \Cref{prop:price}\eqref{price:up}, since $x\in B_{m_0}^{\bfY_{f_i},k}$ by \eqref{fi>t}.}
    \DimaA{This proves} \eqref{lem:cont.bnd:S2}.
        \item\label{lem:cont.bnd:7} The rectification of $\bfX$ is obtained from a union of rectifications of each $\gamma^{-1}(\bfU_i)$ where $\bfY=\bigcup \bfU_i$ given by the rectification of $\bfY$:\\
        \RamiE{Part \eqref{lem:cont.bnd:S1}} follows immediately from \RamiE{Case \ref{lem:cont.bnd:5}}.     

        \RamiE{For part \eqref{lem:cont.bnd:S2} denote $\bfV_i:=\gamma^{-1}(\bfU_i)$ for all $i$.      
        Let $m_1$ be such that for all $i$ we have $\gamma(B^{\bfV_i,k}_{m})\subset B^{\bfU_i,k}_{m_1}$. By \Cref{prop:price}\eqref{price:up} we can find $m_2$ such that for all $i$ and any $y\in B^{\bfU_i,k}_{m_1}$ we have 
\begin{equation}\label{lem:cont.bnd:7:a}
         B^{\bfU_i,k}_{-m_2}(y)\subset B^{\bfY,k}_{-m}(y).
        \end{equation}
        By Case \ref{lem:cont.bnd:5}   we can find $m_3$ s.t.  for all $i$ and any $x\in B^{\bfV_i,k}_{m}$ we have  \begin{equation}\label{lem:cont.bnd:7:b}
        \gamma(B^{\bfV_i,k}_{-m_3}(x))\subset B^{\bfU_i,k}_{-m_2}(\gamma(x)).
        \end{equation}
        By \Cref{cor:price.no.af}\eqref{price.no.af:down} there is $m'$ s.t. 
        for any 
        $x\in B_{m}^{\bfX,k}$ we have         \begin{equation}\label{lem:cont.bnd:7:c}
        B_{-m'}^{\bfX,k}(x)\subset 
        \bigcup_{i \text{ s.t. } x\in B_{m}^{\bfV_i,k}} 
        B_{-m_3}^{\bfV_i,k}(x).
        \end{equation}
        }
        \RamiE{Finally
\begin{align*}
\gamma(B^{\bfX,k}_{-m'}(x))&\overset{\eqref{lem:cont.bnd:7:c}}{\subset}
\gamma(\bigcup_{i \text{ s.t. } x\in B_{m}^{\bfV_i,k}} 
        B_{-m_3}^{\bfV_i,k}(x))\\
        &=
        \bigcup_{i \text{ s.t. } x\in B_{m}^{\bfV_i,k}} \gamma(
        B_{-m_3}^{\bfV_i,k}(x))\\
&\overset{\eqref{lem:cont.bnd:7:b}}{\subset}
\bigcup_{i \text{ s.t. } x\in B_{m}^{\bfV_i,k}} B^{\bfU_i,k}_{-m_2}(\gamma(x))\\
        &\overset{\eqref{lem:cont.bnd:7:a}}{\subset}
B^{\bfY,k}_{-m}(\gamma(x))
\end{align*}        
        }
        \item  \label{lem:cont.bnd:8} The rectification of $\bfX$ is simple, $\gamma$ is an isomorphism.\\
        Let $\bfY_1$ be identical to  $\bfY$ as a variety but with a rectification defined by a cover by principal open sets such that the identity map $\bfY_1\to \bfY$ satisfies the condition of the previous case (Case \ref{lem:cont.bnd:7}). The statement follows now from the previous 2 cases \RamiE{as we can decompose $\gamma$ as $\bfX\to\bfY_1\to\bfY$}.
    \item \label{lem:cont.bnd:9}$\gamma$ is an isomorphism and the rectification of  $\bfY$ is obtained from a union of rectifications of each $\gamma(\bfU_i)$ where $\bfX=\bigcup \bfU_i$ given by the rectification of $\bfX$:\\
    Follows from the previous case (Case \ref{lem:cont.bnd:8})  \RamiE{in the same way as the way Case \ref{lem:cont.bnd:7}  follows from Case \ref{lem:cont.bnd:5}}.
    \item \label{lem:cont.bnd:10}$\gamma$ is an isomorphism.\\
    Follows from Cases \ref{lem:cont.bnd:7} and \ref{lem:cont.bnd:9}.
    \item general case:\\
    Follows from Cases \ref{lem:cont.bnd:7} and \ref{lem:cont.bnd:10}.
\end{enumerate}    
\end{proof}

\begin{cor}\label[cor]{lem:indep-rect}    
    Let $\bfX_1, \bfX_2$ be two copies of the same $\F_\ell$-variety with two (possibly different) rectifications. Let $\bfZ\subset \bfX_1$  be a closed subvariety. Then 
    
\begin{enumerate}[(i)]
\item\label{lem:indep-rect:1}    
    for any $m\in \N$ there is  $m' \in \N$ such that for any $k\in \N$ we have:
    \begin{enumerate}
        \item\label{lem:indep-rect:1a} $B_m^{\bfX_1,k}\subset B_{m'}^{\bfX_2,k}$
        \item for any $x\in \RamiE{B_m^{\bfX_1,k}}$ we have $B_{\RamiE{-m'}}^{\bfX_1,k}(x)\subset B_{\RamiE{-m}}^{\bfX_2,k}(x)$
        \item\label{lem:indep-rect:1c} $B_{\RamiE{-m'}}^{\bfX_1,k}(\bfZ) \RamiE{\cap B_m^{\bfX_1,k}}\subset B_{\RamiE{-m}}^{\bfX_2,k}(\bfZ)$.
    \end{enumerate}    
            \item \label{lem:indep-rect:2} 
            For any \RamiB{$\mu$-rectification}s of $\bfX_i$  and  $m\in \N$, there exists $m'$ such that for any $k$ we have:  $$\mu_m^{\bfX_1,k}< \ell^{km'}\mu_{m'}^{\bfX_2,k}.$$
    \end{enumerate}    
\end{cor}
\begin{proof}
Items \eqref{lem:indep-rect:1}  follow immediately from the previous proposition (\Cref{lem:cont.bnd}) \RamiE{and \Cref{lem:tri}}. It is enough to prove \eqref{lem:indep-rect:2}. We will proceed by analyzing cases:
\begin{enumerate}[{Case} 1.]
    \item The \RamiB{$\mu$-rectification}s of $\bfX_i$ are simple, and their embedding into an affine space is the same:\\
    Let $\omega_i$ be the form on $\bfX_i$. Let $g=\frac{\omega_1}{\omega_2}\in O^\times(\bfX_1)$.
    Fix $m$. By \Cref{lem:cont.bnd}\eqref{lem:cont.bnd:S1} there is $m'>m$ such that  for any $k\in \N$ we have $$\max(\val(g(B_m^{\bfX_1,k}))) <m'.$$
    Now we have:
    $$\mu_m^{\bfX_1,k}=\mu_m^{\bfX_2,k} |g|< \ell^{km'}\mu_{m}^{\bfX_2,k}\leq \ell^{km'}\mu_{m'}^{\bfX_2,k}.$$
    \item The \RamiB{$\mu$-rectification}s of $\bfX_i$ are simple and the forms on $\bfX_i$ are the same:\\
    follows immediately from \Cref{lem:cont.bnd}\eqref{lem:cont.bnd:S1}.
    \item The \RamiB{$\mu$-rectification}s of $\bfX_i$ are simple:\\
    Let $\bfX_3$ be a  variety identical to $\bfX_1,\bfX_2$ with a rectification given by the embedding of $\bfX_1$ to an affine space and the form on $\bfX_2$. The assertion follows now from the previous cases (applied to the pairs $(\bfX_1,\bfX_3)$ and $(\bfX_3,\bfX_2)$  in correspondence).  
    \item The covers of $\bfX_1$ and $\bfX_2$ are the same:\\
    Follows from the previous case.
        \item There is an invertible  top form $\omega$ on $\bfX_1$  such that the  forms in  the rectifications  on  both $\bfX_1$ and $ \bfX_2$ are  restrictions of $\omega$.\\
        Let $M$ be the size of largest of the 2 covers of $\bfX_i$. We have $$|\omega_{\F_{\ell^k}((t))}|1_{B_m^{\bfX_i,k}}\leq \mu_m^{\bfX_i,k}\leq M |\omega_{\F_{\ell^k}((t))}|1_{B_m^{\bfX_i,k}}.$$
        The assertion follows now from part \eqref{lem:indep-rect:1}\eqref{lem:indep-rect:1a}.
    \item\label{lem:indep-rect:c6} $\bfX_1$ admits an invertible top-form:\\
    Follows from the previous 2 cases.
    \item The rectification of $\bfX_1$ is obtained by rectifications of each $\bfU_i$ where $\bfX_2=\bigcup \bfU_i$ is the cover of $\bfX_2$ given by its rectification:\\
    Follows from the previous case.
        \item The rectification of $\bfX_2$ is obtained by a rectification of each $\bfU_i$ where $\bfX_1=\bigcup \bfU_i$ is the cover of $\bfX_1$ given by its rectification:\\
    Follows from Case \ref{lem:indep-rect:c6}.
    \item General case:\\
    Follows from the last 2 cases.
\end{enumerate}    
\end{proof}
At some point we will need the following stronger version of \Cref{lem:indep-rect}
\eqref{lem:indep-rect:1}\eqref{lem:indep-rect:1a}:
\begin{lem}\label{lem:kot.bnd}
    Let $\bfX_1, \bfX_2$ be two copies of the same $\F_\ell$-variety with two (possibly different) rectifications. 
Then there exists $a\in \N$ such that 
    for any $m,k\in \N$ we have:
$$B_m^{\bfX_1,k}\subset B_{am+a}^{\bfX_2,k}$$
\end{lem}
\begin{proof}
    In fact, the proof of \Cref{lem:cont.bnd} provides a proof of this lemma. It also follows from \cite[Proposition 18.1 (1)]{Kot}.
\end{proof}

\begin{cor}\label[cor]{cor:cls.embd}
    Let $\gamma: \bfX\to \bfY$ be a closed embedding. Then for any integer $m\in \N$   there exists $m'\in \N$ such that for any $k\in \N$ we have  $\gamma^{-1}(B^{\bfY,k}_m)\subset B^{\bfX,k}_{m'}$
\end{cor}
\begin{proof}
    Let $\bfX'$ be the variety \RamiD{$\bfX$} with the induced rectification from $\bfY$. We have 
    $\gamma^{-1}(B^{\bfY,k}_m)= B^{\bfX',k}_{m}$. The assertion follows now from \Cref{lem:indep-rect}\eqref{lem:indep-rect:1a}.
\end{proof}

\begin{lem}\label{lem.sub.bnd}
    Let $\gamma: \bfX\to \bfY$ be  a submersion of \RamiB{$\mu$-rectified} varieties. Then for any $m$ there is $m'$ such that $$\gamma_*(\mu_m^{\bfX,k})\RamiK{\leq} \ell^{km'}\mu_{m'}^{\bfY,k}.$$
\end{lem}
\begin{proof}
$ $
    \begin{enumerate}[{Case} 1.]    
        \item The rectifications of $\bfX$ and of $\bfY$ are simple, $\gamma$ is \'etale and the form on $\bfX$ is the pullback of the form on $\bfY$:\\
        Fix $m\in\N$.
        By \cite[\href{https://stacks.math.columbia.edu/tag/03JA}{Tag 03JA}, \href{https://stacks.math.columbia.edu/tag/03J5}{Tag 03J5}]{SP} there is $M\in \N$ such that for any $y\in \bfY(\bar\F_\ell((t)))$ we have $$M> \# \gamma^{-1}(y).$$        
        By \Cref{lem:cont.bnd} there exists $m_1>m$ such that $$\gamma(B_m^{\bfX,k}) \subset B_{m_1}^{\bfY,k}.$$
        Take $m'=m_1M$. Let $k\in \N$.
        Let $\omega_\bfY$ be the form on $\bfY$ given by its \RamiB{$\mu$-rectification}. For $k\in\N$, let $\omega_{\bfY,k}:=(\omega_{\bfY})_{\F_{\ell^{\RamiJ{k}}}((t))}$
        We have $$\gamma_*(\mu_{m}^{\bfX,k})=f\cdot |\omega_{\bfY,k}|$$ where $$f(y)=\#\{x\in B_{m}^{\bfX,k}|\gamma(x)=y\}.$$
        So, 
        \begin{align*}
        \gamma_*(\mu_{m}^{\bfX,k})=f\cdot |\omega_{\bfY,k}|\leq M \cdot 1_{\gamma(B_m^{\bfX,k})} \cdot |\omega_{\bfY,k}|\leq  M \cdot 1_{B_{m_1}^{\bfY,k}} \cdot |\omega_{\bfY,k}|=M \cdot \mu_{m_1}^{\bfY,k}<
        \ell^{\RamiJ{k}m'} \cdot \mu_{m'}^{\bfY,k}
        \end{align*}
        
        \item\label{lem.sub.bnd:2} the rectifications of $\bfX$ and of $\bfY$ are simple and $\gamma$  is \'etale:\\
        Follows from the previous case using \Cref{lem:indep-rect}\eqref{lem:indep-rect:2}.
        \item\label{lem.sub.bnd:3}  the \RamiB{$\mu$-rectification}s of $\bfX$ and of $\bfY$ are simple, $\bfX=\bfY\times \A^M$, and $\gamma$  is the projection:\\
        WLOG assume that the \RamiB{$\mu$-rectification} of $\bfX$ is given by the \RamiB{$\mu$-rectification} of $\bfY$ and the standard \RamiB{$\mu$-rectification} of $\bA^M$. Take $m'=mM+1$. The assertion is a straightforward computation.
        \item  $\gamma=\gamma_1\circ \gamma_2$  \DimaA{where $\gamma_1$ satisfies the conditions of Case \ref{lem.sub.bnd:3} and $\gamma_2$ satisfies} the conditions of Case \ref{lem.sub.bnd:2}:\\
        Follows immediately from the previous 2 cases.
        \item  The rectification of $\bfY$ is simple:\\
        By \cite[\href{https://stacks.math.columbia.edu/tag/039P}{Tag 039P}]{SP} there is a cover $\bfX=\bigcup_{i=1}^M \bfU_i$ such that $\gamma|_{\bfU_i}$  satisfy the condition of the previous case.  WLOG we can also assume that each $\bfU_i$  admits a simple \RamiB{$\mu$-rectification}. By \Cref{lem:indep-rect}, we can also assume that the \RamiB{$\mu$-rectification} of $\bfX$ is coming from these rectifications.        
        The statement follows now from the previous case.     
        \item  General case:\\
        Follows from the previous case and \Cref{lem:indep-rect}.       
    \end{enumerate}
\end{proof}
\begin{cor}\label[cor]{cor:bnd.int}
    Let $\bfX$ be a \RamiB{$\mu$-rectified} variety. Then for any $m\in \N$ there exists $M\in \N$ such that for any $k\in \N $:
    $$\mu_m^{\bfX,k}(B_\infty^{\bfX,k})< \ell^{kM}.$$
\end{cor}
\begin{proof}
    It follows from the previous lemma for the map $\gamma:\bfX\to pt\RamiE{:=}\bA^0$.
\end{proof}
\section{Effective version of the implicit function theorem and its corollaries}\label{ssec:imp}
In this subsection we prove \RamiB{\Cref{thm:main.et.mono} (see \Cref{lem:et.mono} below)  and \Cref{thm:main.imp} (see \Cref{lem:eff.imp} below)}.

We also deduce 2 statements:
\begin{itemize}
    \item \Cref{prop: nbd.of.subvar}. See \Cref{lem:dev}.
    \item The main part of \Cref{thm: Eff. bounds on push of smooth measures}: 
    The push of the measure $\mu_{m}^{\bfX,k}$ under a submersion $\bfX\to \bfY$  controls (from above) the measure $\mu_{m}^{\bfY,k}$ on the support of the former. See \Cref{lem:bnd.on.sup}.
\end{itemize}

We start with the following:
\begin{thm}\label[thm]{lem:et.mono}
Let $\gamma:\bfX\to \bfY$  be  an \'etale map of smooth (rectified) algebraic varieties defined over $\F_\ell$.  Then for any $m$ there is $m'$ such that for any $k$ and any $x\in B^{\bfX,k}_m$ the map $\gamma|_{B^{\bfX,k}_{-m'}(x)}$ is a monomorphism.
\end{thm}
\begin{proof}$ $
\RamiI{The stament does not depend on the rectification on $\bfY$. Also,}
by \Cref{lem:indep-rect} the statement does not depend on the rectification \RamiI{on $\bfX$.} So we  will choose \RamiI{them} each time as convenient.
For an integer $n$, denote by $\bfY_n$  the collection of monic polynomials of degree $n$, considered as an affine space. Denote also $\bfX_n:=\{(f,a)|f \in \bfY_n; f(a)=0, f'(a)\neq 0\}$.
\begin{enumerate}[{Case} 1.]
    \item  $\bfX=\bfX_n$, $\bfY=\bfY_n$ (for some $n$), and $\gamma$ is the projection:\\
    \RamiC{E}mbed $\bfX$ into $\A^{n+2}$ by $(f,a)\mapsto (f,a, \frac{1}{f'(a)})$. This gives a rectification of $\bfX$. Take also the rectification of $\bfY$ that comes from the fact that it is an affine space.
    
    Set $m'=mn+1+m+1$.
Let $x=(f,a)\in B^{\bfX,k}_m$. It is enough to show that if $(f,b)\in B^{\bfX,k}_{-m'}(x)$ then $a=b$. 
For a polynomial $f\in \Z[x]$ denote $f^{[k]}:=\frac{f^{(k)}}{k!}$. This is a polynomial over $\Z$, so this operation is defined over any field.
We have $$0=f(b)=f(a+(b-a))=f(a)+f^{[1]}(a)(b-a)+\cdots=f^{[1]}(a)(b-a)+\cdots$$
    Assuming that $b\neq a$ we obtain:
    $$f^{[1]}(a)+(b-a)f^{[2]}(a)+\dots=0$$
Since $x\in B^{\bfX,k}_m$, we have $$|f^{[i]}(a)|\leq |t^{-mn}|.$$ Thus $$|(b-a)f^{[2]}(a)+\dots|< |t|^{m+1}.$$ On the other hand $$\RamiK{|}f^{[1]}(a)\RamiK{|}=\RamiK{|}f'(a)\RamiK{|}\geq |t|^{m}.$$   This leads to a contradiction.  
    \item $\gamma$ is standard \'etale map\footnote{See \cite[\href{https://stacks.math.columbia.edu/tag/02GI}{Definition 02GI}]{SP}.}:\\
    \RamiC{Follows} from the previous step\RamiH{, \Cref{cor:price.no.af},
    \Cref{lem:indep-rect},} \Cref{lem:cont.bnd} and the fact that we have a \RamiH{commutative} diagram:
    $$ 
\begin{tikzcd}     \RamiH{\bfX}\arrow[dr,"\RamiH{\gamma}"']\arrow[r,hook,"\RamiH{i}"] &\bfX'\arrow[dr, phantom, "\square"] \arrow[r] \arrow[d] & \bfX_n \arrow[d] \\
     &\bfY \arrow[r]  & \bfY_n
\end{tikzcd}
$$ 
\RamiH{where $i$ is an open embedding.}
    \item \RamiC{General} case:\\
    Follows from the previous case using  \cite[\href{https://stacks.math.columbia.edu/tag/02GT}{Tag 02GT}]{SP} \RamiE{in view of \Cref{cor:price.no.af}
    and \Cref{lem:indep-rect}}.
\end{enumerate}    
\end{proof}
The next result is an effective version of the open mapping theorem:
\begin{thm}\label[thm]{lem:eff.imp}
Let $\gamma:\bfX\to \bfY$ be a smooth map of smooth \Rami{(rectified)} algebraic varieties.  Then for any $m\in \N$ there is $m'>m$ such that for any $k$ and any $x\in B^{\bfX,k}_m$ we have $$\gamma(B^{\bfX,k}_{-m}(x))\supset B^{\bfY,k}_{-m'}(\gamma(x)).$$
\end{thm}
\RamiE{
\begin{proof}[Idea of the proof]
The classical proof of the inverse function theorem for analytic functions is rather explicit for affine spaces. So a careful analysis of this proof provides a proof of \Cref{lem:eff.imp} for the case of affine spaces (and even slightly beyond that). This is done in Case 1 below. Though it is tedious it does not contain essentially new ingredients.

However, in the classical case, the reduction to the affine case is trivial, and for \Cref{lem:eff.imp} this reduction is not trivial at all.

Our main idea to perform this reduction is to break it down to simpler reductions that can be easily performed based on the methods of \S\ref{ssec:ballsMe}. This breakdown can be read from the titles of the cases below.
\end{proof}
}

\begin{proof}
$ $
\begin{enumerate}[{Case} 1.]
    \item $\bfX= (\bA^d)_f$ is an almost affine space,  $\bfY=\bA^d$:\\
    Let $j\in O(\bfX)$ be the Jacobian of $\gamma$\RamiH{, and write $j^{-1}=q/f^{L}$ with $q\in \F_\ell[x]$ and $L\in \RamiJ{\Z_{>0}}$}. Choose the rectification of $\bfX$ that is given by the embedding $x\mapsto (x,f(x)^{-1},j(x)^{-1})$, \Rami{and choose the standard \RamiE{rectification} of $\bfY$}. Fix $m$. Write $$\gamma(x)=\frac{h(x)}{f(x)^M}=\frac{\sum_{|i|\leq \RamiH{N}} a_i x^i}{f(x)^M},\qquad \RamiH{N:=\deg(h)},$$
    in multi-index notation, where $a_i$ are vectors \RamiJ{and $M\in\Z_{>0}$}.  Note that $\val(a_i)\RamiK{\geq}0$, where the valuation of a vector is defined to be the minimum of the valuations of its coordinates. 
    Let $$\RamiH{C_1:=m\left(2\deg(f)M+2N+2M+2\right),\qquad C_2:=4md\left(\deg(f)M+N+2M\right)}$$
    and $$m'=\RamiH{C_1+2C_2+m\left(\RamiJ{3+\deg(f)+}\deg(q)+L\deg(f)+2L\right)+Mm+1}.$$ Take $y'\in B^{\bfY,k}_{-m'}(\gamma(x))$. We have to find $x'\in  B^{\bfX,k}_{-m}(x)$ such that $\gamma(x')=y'$. Define recursively a sequence $x_r \in B^{\bfX,k}_{-m}(x)$. Set $x_0=x$, and define $$x_{r+1}:= x_r+(D_{x_r}\gamma)^{-1}(y'-\gamma(x_r)).$$
    We will show by induction \RamiC{on $r$} that \RamiC{for any $r\geq 0$}:
    \begin{enumerate}[$(a_r)$]
        \item \label{lem:eff.imp:1}
        $\val(y'-\gamma(x_r))\RamiC{\geq}m'+r$
        \item\label{lem:eff.imp:3} $x_{r}\in B_{m}^{\bfX,k}$
        \item\label{lem:eff.imp:2} $x_{r+1}\in B_{-m-r}^{\bfX,k}(x_r)$
    \end{enumerate}    
    Before proving these statements, we will show that for any given $r$ statements 
    \RamiC{$(\ref{lem:eff.imp:1}_r)$}
    and \RamiC{$(\ref{lem:eff.imp:3}_r)$}  imply \RamiC{$(\ref{lem:eff.imp:2}_r)$}. 
For this note that for $w\in \F_{l^k}((t))^d$ we have 
\RamiC{
\begin{align*}        \val\left((D_{x_r}\gamma)^{-1}w\right)
    &= \val\left(\det\left(D_{x_r}\gamma\right)^{-1}Adj\left(D_{x_r}\gamma\right)w\right)
    = \val\left(j(x_r)^{-1}Adj\left(D_{x_r}\gamma\right)w\right)
    \\
    &\geq
    \val\left(j(x_r)^{-1}\right)
    + \val(w) + (d-1)\min_{1\leq s,t\leq d} \val\left((D_{x_r}\gamma)_{s,t}\right)
    \\
    &=        \val\left(j(x_r)^{-1}\right)
    + \val(w) + (d-1)\min_{1\leq s,t\leq d} \val\left(\partial_t\gamma_s (x_r)\right)
    \\
    &=     \val\left(j(x_r)^{-1}\right)
    + \val(w)+ (d-1)\min_{1\leq s,t\leq d} \val \left( \partial_t\left (\tfrac{h_s}{f^{M}}\right)(x_r) \right)   
    \\
    &=
    \val\left(j(x_r)^{-1}\right)
    + \val(w) + (d-1)\min_{1\leq s,t\leq d} \val\left(\tfrac{f^{M}(x_r) \partial_t h_s(x_r)- \partial_t(f^M)(x_r) h_s(x_r)}{f(x_r)^{2M}}\right)    
    \\
    &=
    \val\left(j(x_r)^{-1}\right)
    + \val(w) +\\
    &+ (d-1)\min_{1\leq s,t\leq d}\left( \val\left(f(x_r)^{M}\partial_t h_s(x_r)- \partial_t(f^M)(x_r) h_s(x_r)\right)     -  \val\left(f(x_r)^{2M}\right)\right) 
        \\&\mathrel{\RamiH{\geq}}    \val\left(j(x_r)^{-1}\right)
    + \val(w) +\\
    & +(d-1)\min_{1\leq s,t\leq d}\left( \min(\val\left(f(x_r)^{M}\partial_t h_s(x_r)\right),\val\left( \partial_t(f^M)(x_r) h_s(x_r)\right) \right)  \\&\RamiH{-(d-1)  \val\left(f(x_r)^{2M}\right)} 
\end{align*}

Statement \RamiC{$(\ref{lem:eff.imp:3}_r)$} implies that:
\begin{align*}        &
 \val\left(j(x_r)^{-1}\right)
    + \val(w) + (d-1)\min_{1\leq s,t\leq d}\left( \min(\val\left(f(x_r)^{M}\partial_t h_s(x_r)\right),\val\left( \partial_t(f^M)(x_r) h_s(x_r)\right) \right)  \\&\RamiH{-(d-1)  \val\left(f(x_r)^{2M}\right)} 
    \geq
-m-d(\deg(f) M+\RamiH{N})m-2dMm \RamiH{+\val(w)}
    \geq
\RamiH{-C_2+\val(w)}
    \end{align*}}
    So 
    \begin{equation}\label{lem:eff.imp:dga}
    \val((D_{x_r}\gamma)^{-1}w)\geq \RamiH{-C_2} +\val(w).    
    \end{equation}
    
    This implies:
            \begin{align*}        
        \val(x_{r+1}-x_{r})=\val((D_{x_{r}}\gamma)^{-1}(y'-\gamma(x_{r})))\geq \RamiH{-C_2}+\val(y'-\gamma(x_{r}))       
    \end{align*}
    So, by \RamiC{$(\ref{lem:eff.imp:1}_r)$} we obtain:
    \begin{equation}\label{lem:eff.imp:x}            
    \val(x_{r+1}-x_{r})\geq \RamiH{-C_2}+m'+r\geq 3m+m\deg(f)+m\RamiH{\left(\deg(q)+L\deg(f)+2L\right)}+r
    \end{equation}
    This implies that 
    \begin{equation}
    \val(f(x_{r+1})-f(x_{r}))\geq \val(x_{r+1}-x_{r})-m\deg(f)\geq 3m+r.
    \end{equation}
    So $$\val(f(x_{r+1}))=\val((f(x_{r+1})-f(x_r))+f(x_r))\leq m.$$
    Thus    
    \begin{equation}\label{lem:eff.imp:f}            
\val(\frac{1}{f(x_{r+1})}-\frac{1}{f(x_{r})})=
    \val(\frac{f(x_{r})-f(x_{r+1})}{f(x_{r+1})f(x_{r})})\mathrel{\RamiH{\geq}} m+r.\end{equation}
Similarly, 
        \begin{equation}\label{lem:eff.imp:j}            
        \val(\frac{1}{j(x_{r+1})}-\frac{1}{j(x_{r})})        
    \mathrel{\RamiH{\geq}} m+r.\end{equation}
    Formulas \eqref{lem:eff.imp:x},\eqref{lem:eff.imp:f} and \eqref{lem:eff.imp:j} (which are proven under the assumptions of \RamiC{$(\ref{lem:eff.imp:1}_r)$} and \RamiC{$(\ref{lem:eff.imp:3}_r)$}) imply \RamiC{$(\ref{lem:eff.imp:2}_r)$}.

    We now proceed to prove by induction \RamiC{$(\ref{lem:eff.imp:1}_r)$}, \RamiC{$(\ref{lem:eff.imp:3}_r)$} and \RamiC{$(\ref{lem:eff.imp:2}_r)$}. The \RamiC{base statements  \RamiC{$(\ref{lem:eff.imp:1}_0)$},  and \RamiC{$(\ref{lem:eff.imp:3}_0)$} are obvious}.  So the base statement \RamiC{$(\ref{lem:eff.imp:2}_0)$} follows from the above. For the induction step we assume  \RamiC{$(\ref{lem:eff.imp:1}_r)$}, \RamiC{$(\ref{lem:eff.imp:3}_r)$} and \RamiC{$(\ref{lem:eff.imp:2}_r)$} for an integer  $r$ and prove \RamiC{\RamiC{$(\ref{lem:eff.imp:1}_{r+1})$}, \RamiC{$(\ref{lem:eff.imp:3}_{r+1})$} and \RamiC{$(\ref{lem:eff.imp:2}_{r+1})$}.}

    \RamiC{{By \Cref{lem:tri}\eqref{lem:tri:mon} statement  $(\ref{lem:eff.imp:3}_{r+1})$} follows from  \RamiC{$(\ref{lem:eff.imp:2}_r)$} and \RamiC{$(\ref{lem:eff.imp:3}_r)$}}. Also, as shown before,  \RamiC{$(\ref{lem:eff.imp:1}_{r+1})$} and \RamiC{$(\ref{lem:eff.imp:3}_{r+1})$} imply \RamiC{$(\ref{lem:eff.imp:2}_{r+1})$}. It is left to show 
    \RamiC{$(\ref{lem:eff.imp:1}_{r+1})$} \RamiC{and we can use {$(\ref{lem:eff.imp:3}_{r+1})$}  for this}.  
    For  $z\in \F_{l^k}((t))^d$, write $$\gamma(x_r+z)= \gamma(x_r)+(D_{x_r}\gamma)z+\frac{\sum_{2\leq |i|\leq \RamiH{\deg(f)M+N+1}} b_i z^i}{f(x_r+z)^M}.$$ 
Let $\delta(\RamiC{w}):=f(\RamiC{w})^M(\gamma(\RamiC{w})-\gamma(x_r)- (D_{x_r}\gamma)(\RamiC{w}-x_r)).$ Note that 
\begin{itemize}
    \item this is a polynomial map.
        \item The valuation of its coefficients are bounded from below by $\RamiH{-m(\deg(f)M+N+2M+1)}$.
        \item Its degree is bounded by $\deg(f) M +\RamiH{N+1}$.
\end{itemize} 
Define $\delta^{[i]}$ as in the proof of \Cref{lem:et.mono}, but for vector valued functions of a vector variable.    
    We have
    \begin{align*}        \val(b_i)=\val(\delta^{[i]}(x_r))\geq \RamiH{-C_1}.
    \end{align*}
    Setting $$z:=(D_{x_r}\gamma)^{-1}(y'-\gamma(x_r)),$$
    formula \eqref{lem:eff.imp:dga} above implies 
$$\val(z)\geq \RamiH{-C_2} +\val(y'-\gamma(x_r)).$$
    So,

\begin{align*}        
\val\left(y'-\gamma(x_{r+1})\right)&=
\val\left(y'-\gamma(x_r+z)\right)=
\\&=
\val\left(y'-
\left(\gamma(x_r)+(D_{x_r}\gamma)z+\frac{\sum_{2\leq |i|\leq \RamiH{\deg(f)M+N+1}} b_i z^i}{f(x_r+z)^M}\right)\right)
=
\\&=
\val\left(y'-
\gamma(x_r)-(D_{x_r}\gamma)z-\frac{\sum_{2\leq |i|\leq \RamiH{\deg(f)M+N+1}} b_i z^i}{f(x_r+z)^M}\right)
=
\\&=
\val\left(\frac{\sum_{2\leq |i|\leq \RamiH{\deg(f)M+N+1}} b_i z^i}{f(x_r+z)^M}\right)
=
\\&=
\val\left(\sum_{2\leq |i|\leq \RamiH{\deg(f)M+N+1}} b_i z^i\right)-\val\left(f(x_r+z)^M\right)
\\&=
\val\left(\sum_{2\leq |i|\leq \RamiH{\deg(f)M+N+1}} b_i z^i\right)-\val\left(f(x_{r+1})^M\right)
\\&\geq
\min_{2\leq |i|\leq \RamiH{\deg(f)M+N+1}} \left(\val(b_i)+|i| \val(z)\right)-\val\left(f(x_{r+1})^M\right)
\\&\geq
\RamiH{-C_1} +2 \val(z)-M\val\left(f(x_{r+1})\right)
\\&\geq
\RamiH{-C_1-2C_2}+2\val\left(y'-\gamma(x_r)\right) -
M\val\left(f(x_{r+1})\right).
\end{align*}

By \RamiC{$(\ref{lem:eff.imp:3}_{r+1})$}, we know that $$\val({f(x_{r+1})^{-1}})\geq -m.$$
By \RamiC{$(\ref{lem:eff.imp:1}_r)$}, we know that $$\val(y'-\gamma(x_r))\geq m'+r.$$
Thus
    \begin{align*}        \val(y'-\gamma(x_{r+1}))&\geq
\RamiH{-C_1-2C_2}
+2m' +2r - Mm\geq
\\&\geq
m'+r+1,
    \end{align*}
    proving \RamiC{$(\ref{lem:eff.imp:1}_{r+1})$} and finishing the proof of \RamiC{$(\ref{lem:eff.imp:1}_r)$}, \RamiC{$(\ref{lem:eff.imp:3}_r)$} and \RamiC{$(\ref{lem:eff.imp:2}_r)$ for every $r$}.

    Now, \RamiC{$(\ref{lem:eff.imp:2}_r)$} implies that $\val(x_{r+1}-x_r)\geq m+r$. Thus the sequence $x_\RamiC{r}$ converges. Let $x':=\lim_{\RamiC{r}\to \infty} x_\RamiC{r}.$ Using \Cref{lem:tri}\eqref{lem:tri:ultra} and \RamiC{$(\ref{lem:eff.imp:2}_r)$} again, we get that \RamiC{$x_r\in B_{\RamiC{-m}}^{\bfX,k}(x)$ for any $r$. Thus  $x'\in B_{\RamiC{-m}}^{\bfX,k}(x)$.} Finally, \RamiC{$(\ref{lem:eff.imp:1}_r)_{r \in \N}$}  implies that $\gamma(x')=y'$ as required.
    \item $\bfX$ and $\bfY$ are almost affine spaces and $\dim(\bfX)=\dim(\bfY)$:\\
        Follows immediately from the previous case.    
    \item\label{case:3.alm} $\bfX$ and $\bfY$ are almost affine spaces, and there exists a \RamiC{morphism} $\RamiC{p:\bfX}\to \bA^{\dim \bfX-\dim \bfY}$ such that the map $\RamiC{\gamma\times p:}\bfX\to \bfY\times \bA^{\dim \bfX-\dim \bfY}$ is \'etale:\\
        Follows from the previous case applied to the map $\RamiC{\gamma\times p}$.
    \item\label{case:4.cov} $\bfX$ \RamiC{is an open subset of an affine space} and $\bfY$ \RamiC{is} \RamiE{an} almost affine \RamiE{space}:\\
        We can find a finite cover  $\bfX:=\bigcup \bfU_i$ such that  $\gamma|_{\bfU_i}$ satisfies the conditions of 
        the previous case.                 \RamiE{By the previous case we can choose $m'$ s.t. for all $i$ and for every $x\in B_m^{\bfU_i,k}$  we have $$\gamma(B^{\bfU_i,k}_{-m-m s_\bfX}(x))\supset B^{\bfY,k}_{-m'}(\gamma(x)).$$
        
        Take $x\in B_m^{\bfX,k}$. Find $i$ s.t. $x\in B_m^{\bfU_i,k}$. Now we have        $$\gamma(B_{-m}^{\bfX,k}(x)) \overset{\text{Prop \ref{prop:price}\eqref{price:up}}}{\supset} \gamma(B^{\bfU_i,k}_{-m-m s_\bfX}(x))\supset B^{\bfY,k}_{-m'}(\gamma(x)).$$}         
    \item $\bfY$ is an almost affine space and $\bfX$ is a 
 fiber of a smooth map between almost affine spaces:\\
 Write $\bfX$ as the fiber of a smooth map $\bfU\to \bfV$.
 \RamiJ{Shrinking $\bfU$ we can} \RamiC{extend the map $\gamma$ to a map $\gamma':\bfU\to \bfY$, this defines a map $\gamma'':\bfU\to \bfV\times \bfY$. Note that $\gamma''$ is smooth at $\bfX$. Thus we can find \DimaA{an open subset} $\bfU'\subset \bfU$ containing $\bfX$ such that $\gamma''|_{\bfU'}$ is smooth.}
 The statement follows now from the previous case applied to  \Rami{$\gamma''|_{\bfU'}$}. 
    \item $\bfY$ is an almost affine space:\\
            \RamiF{First let us show that we}  can find a  finite cover  $\bfX:=\bigcup \bfU_i$ such that  $\gamma|_{\bfU_i}$ satisfies the conditions of 
        the previous case. 
        \RamiF{For this, we can assume that $\bfX$ is affine. Embed $\bfX$ to $\bA^n$ and let $I$ be the defining ideal of $\bfX$ in $\bA^n$. Let $x\in \bfX$. Since $\bfX$ is smooth at $x$ we can find $f_1,\dots f_k \in I$ s.t. $d_xf_i$ are linearly independent, and $f_1,\dots f_k$ generates $I$ in a neighborhood of $x$.  We can also assume that this neighborhood is $\A^n_g$ for some $g\in O(\A^n)$. Now the map $\A^n_g\to \A^k$ given by  $f_1,\dots f_k$ is smooth at $x$. Hence it is smooth at a neighborhood $\A^n_{g_1}$ for some $g_1\in O(\A^n)$. So $\gamma|_{\bfX_{g_1}}$ satisfies the condition of the previous case, and we get the desired cover.}
        
        The statement now follows from the previous case
        \RamiE{in the same way that Case \ref{case:4.cov} followed from  Case \ref{case:3.alm}}.
    \item $\bfY$ is an affine variety that admits an \'etale map to an affine space:\\
Let $\delta: \bfY\to \bfU$ be an \'etale map to an almost affine space. Fix $m$. By \Cref{lem:cont.bnd} there is $m_1>m$  such that for any $k\in\N$  we have    
\begin{equation}\label{lem:eff.imp:bnd}  
\gamma(B^{\bfX,k}_{m})\subset B^{\bfY,k}_{m_1}.
\end{equation}
By \Cref{lem:et.mono}  there is $m_2>m_1$  such that for any $k\in\N$ and any $x\in B^{\bfY,k}_{m_1}$ we have   \begin{equation}\label{lem:eff.imp:mono}
\delta|_{B^{\bfY,k}_{-m_2}(x)}\quad \text{ is a monomorphism.}
\end{equation}
By \Cref{lem:cont.bnd} there is $m_3>m_2$  such that for any $k\in\N$ and any $x\in B^{\bfX,k}_{m}$ we have    
\begin{equation}\label{lem:eff.imp:cont.g}
\gamma(B^{\bfX,k}_{-m_3}(x))\subset B^{\bfY,k}_{-m_2}(\gamma(x)).
\end{equation}
By the previous case there is $m_4>m_3$  such that for any $k\in\N$ and any $x\in B^{\bfX,k}_{m}$ we have    
\begin{equation}\label{lem:eff.imp:open}
\delta(\gamma(B^{\bfX,k}_{-m_3}(x)))\supset B^{\bfU,k}_{-m_4}(\delta(\gamma(x))).
\end{equation}

By \Cref{lem:cont.bnd} there is $m_5>m_4$  such that for any $k\in\N$ and any $x\in B^{\bfY,k}_{m_{\RamiJ{4}}}$ we have    
\begin{equation}\label{lem:eff.imp:cont.d}
\delta(B^{\bfY,k}_{-m_5}(x))\subset B^{\bfU,k}_{-m_4}(\delta(x)).
\end{equation}

Take $m'=m_5$. Let $k$ be an integer and $x\in B^{\bfX,k}_{m}$. Let $y'\in  B^{\bfY,k}_{-m'}(\gamma(x))$. We have to find $x'\in B^{\bfX,k}_{-m}(x)$ such that $\gamma(x')=y'$. 

By \eqref{lem:eff.imp:bnd} $\gamma(x)\in  B^{\bfY,k}_{m_1}.$ 
So, by \eqref{lem:eff.imp:cont.d} 
$$\delta(y')\in B^{\bfU,k}_{-m_4}(\delta(\gamma(x))).$$
Thus, by \eqref{lem:eff.imp:open}, there is $x'\in B^{\bfX,k}_{-m_3}(x)$ such that 
\begin{equation}\label{lem:eff.imp:eq.d} 
\delta(\gamma(x'))=\delta(y').
\end{equation}
By \eqref{lem:eff.imp:cont.g} $$\gamma(x')\in B^{\bfY,k}_{-m_2}(\gamma(x)).$$ Also 
$$y'\in B^{\bfY,k}_{-m_2}(\gamma(x)).$$ Therefore, \eqref{lem:eff.imp:mono} and \eqref{lem:eff.imp:eq.d} imply that $\gamma(x')=y'$ as required.
    \item General case:\\
    We find a finite cover $\bfY=\bigcup \bfU_i$ such that 
    the maps $$\gamma_i:\RamiE{\bfV_i:=}\gamma^{-1}(\bfU_i)\to \bfU_i$$  obtained by the restriction of $\gamma$  to
    $\bfU_i$
    satisfy the conditions of the previous case. 
    \RamiE{
    Using \Cref{lem:indep-rect}\eqref{lem:indep-rect:1}, WLOG we may assume that the rectification on $\bfX$ comes from a rectification of each  $\bfV_i$ and that the rectification on $\bfY$ is given by the open affine cover   $\bfY=\bigcup \bfU_i$.
    By \Cref{cor:price.no.af}\eqref{price.no.af:up} we can find $m_1$ s.t. for all 
    $i$ 
    and for every $x\in B_m^{\bfV_i,k}$  we have $$B^{\bfX,k}_{-m}(x)\supset B^{\bfV_i,k}_{-m_1}(x).$$
    
    By the previous case we can choose $m_2$ s.t. for all $i$ and for every $x\in B_m^{\bfV_i,k}$  we have 
    \begin{equation}\label{lem:eff.imp:gen}
    \gamma(B^{\bfV_i,k}_{-m_1}(x))\supset B^{\bfU_i,k}_{-m_2}(\gamma(x)).
    \end{equation}}
\RamiE{By 
\Cref{lem:cont.bnd}\eqref{lem:cont.bnd:S1}  we can also find $m_3$ s.t. for all $i$ we have:
$$\gamma(B^{\bfV_i,k}_m)\subset B^{\bfU_i,k}_{m_3}$$
}   
       \RamiE{
       Take $m'=m_2+m_3 s_{\bfY} $.
        Let $x\in B_m^{\bfX,k}$. Find $i$ s.t. $x\in B_m^{\bfV_i,k}$. 
Now we have        
\begin{align*}
\gamma(B_{-m}^{\bfX,k}(x)) &\overset{\text{Cor }\ref{cor:price.no.af}\eqref{price.no.af:up}}{\supset} \gamma(B^{\bfV_i,k}_{-m_1}(x))\\
&\overset{\eqref{lem:eff.imp:gen}}{\supset} B^{\bfU_i,k}_{-m_2}(\gamma(x))\\
&\overset{\text{Prop }\ref{prop:price}\eqref{price:down}}{\supset}        
        B^{\bfY,k}_{-m_2-m_3 s_{\bfY} }(\gamma(x))
        =B^{\bfY,k}_{-m'}(\gamma(x)).
\end{align*}
        }

\end{enumerate}
\end{proof}

\begin{prop}\label[prop]{lem:dev}
Let $\bfX$ be a rectified variety.
    Let $\bfU\subset \bfX$ be open and $\bfZ:=\bfX\smallsetminus \bfU$. Then for any $m\RamiC{\in \N}$ there exists $m'\RamiC{>m}$ such that for any $k$ we have $$B^{\bfX,k}_m\subset B^{\bfU,k}_{m'}\cup B^{\bfX,k}_{-m}(\bfZ).$$
\end{prop}
\begin{proof}
$ $
\begin{enumerate}[{Case} 1.] 
    
        \item The rectification of $\bfX$ is simple,
        $\bfZ$ is a fiber of a smooth map $\delta:\bfX\to \bfY$, and $\bfY$ \RamiE{is affine}:\\
        \RamiE{First note that by \Cref{lem:indep-rect}\eqref{lem:indep-rect:1a}, the statement does not depend on the rectification of $\bfU$.}
        Let $x_1,\dots,x_d$ be the coordinates of the affine space that includes $\bfY$.
        WLOG assume that $\bfZ$ is the fiber of $0\in \bfY(\F_\ell)$.
        Let $f_i:=\delta^*(x_i)$. Embed $\bfX_i:=\bfX_{f_i}$ to affine space in the standard way, and choose the rectification of $\bfU$ given  by these embeddings. It is easy to see that for any integers $m',k$ we have $$B^{\bfU,k}_{m'}=B^{\bfX,k}_{m'}\smallsetminus \delta^{-1}(B^{\bfY,k}_{-m'-1}(0)).$$
        Indeed, 
        \begin{align*}                B^{\bfU,k}_{m'}&=\bigcup_i B^{\bfX_i,k}_{m'}=\bigcup_i\{x\in \bfX_i(\F_{\ell^k}((t)))| x\in B^{\bfX,k}_{m'}; \val(f_i(x)^{-1})\geq -m' \}=
        \\&=
        \bigcup_i (B^{\bfX,k}_{m'}\smallsetminus \{x\in \bfX(\F_{\ell^k}((t)))|\val(f_i(x))>m' \})=
        \\&=
                B^{\bfX,k}_{m'}\smallsetminus \bigcap_i  \{x\in \bfX(\F_{\ell^k}((t)))|\val(f_i(x))>m' \}=
        \\&=B^{\bfX,k}_{m'}\smallsetminus \delta^{-1}(B^{\bfY,k}_{-m'-1}(0)).
        \end{align*}
Fix $m$. By \Cref{lem:eff.imp} there exists 
 $m_1\RamiC{>m}$ such that for any $k\RamiC{\in \N}$ and any $x\in B^{\bfX,k}_m$ we have $$\delta(B^{\bfX,k}_{-m}(x))\supset B^{\bfY,k}_{-m_1}(\delta(x)).$$
 Take $m'=m_1\RamiJ{-}1$. It is left show that $$\delta^{-1}(B^{\bfY,k}_{-m'-1}(0))\RamiJ{\cap B_m^{\bfX,k}}=\delta^{-1}(B^{\bfY,k}_{-m_1}(0))\RamiJ{\cap B_m^{\bfX,k}}\subset B^{\bfX,k}_{-m}(\bfZ).$$ Let $x\in \delta^{-1}(B^{\bfY,k}_{-m_1}(0))\RamiJ{\cap B_m^{\bfX,k}}$. We have to show that $x\in B^{\bfX,k}_{-m}(\bfZ)$.
We have $$0\in B^{\bfY,k}_{-m_1}(\delta(x))\subset \delta(B^{\bfX,k}_{-m}(x)).$$ So we have $z\in  B^{\bfX,k}_{-m}(x)$ such that $\delta(z)=0$. By \Cref{lem:tri} this implies that $x\in B^{\bfX,k}_{-m}(z)$. Also $z\in \bfZ(F_{\ell^k}((t)))$. So 
$x\in B^{\bfX,k}_{-m}(B^{\bfX,k}_m\cap \bfZ)$, as required.
        \item $\bfZ$ is smooth.\\
        \RamiE{First note that by \Cref{lem:indep-rect}\eqref{lem:indep-rect:1a} and \eqref{lem:indep-rect:1c} the statement does not depend on the rectification of $\bfX$. Indeed, if $\bfX'$ is the variety $\bfX$ equipped with a different rectification for which we know the assertion of \Cref{lem:dev}, then by \Cref{lem:indep-rect}\eqref{lem:indep-rect:1a}, for any $m$ we can find $m_1$ s.t.
        \begin{equation}\label{lem:dev:eq1}
        B_{m}^{\bfX,k}\subset B_{m_1}^{\bfX',k}.
        \end{equation}
        By \Cref{lem:indep-rect}\eqref{lem:indep-rect:1c} we can find $m_2>m_1$
        s.t.
        \begin{equation}\label{lem:dev:eq2}
        B_{{-m_2}}^{\bfX',k}(\bfZ) {\cap B_{m_1}^{\bfX',k}}\subset B_{{-m_1}}^{\bfX,k}(\bfZ).
        \end{equation}
        Now take $m'$ s.t.
        \begin{equation}\label{lem:dev:eq3}
        B^{\bfX',k}_{m_2}\subset B^{\bfU,k}_{m'}\cup B^{\bfX',k}_{-m_2}(\bfZ).
        \end{equation}
        We obtain:
        \begin{align*}
        B^{\bfX,k}_{m}&\overset{\eqref{lem:dev:eq1}}{=} B^{\bfX',k}_{m_2}\cap B^{\bfX,k}_{m}\subset
        \\&\overset{(\ref{lem:dev:eq1},\ref{lem:dev:eq3})}{\subset}
        (B^{\bfU,k}_{m'}\cup B^{\bfX',k}_{-m_2}(\bfZ)) \cap B^{\bfX',k}_{m_1} \subset
        \\&\subset
        B^{\bfU,k}_{m'}\cup (B^{\bfX',k}_{-m_2}(\bfZ) \cap B^{\bfX',k}_{m_1})\subset
        \\&\overset{\eqref{lem:dev:eq2}}{\subset}
        B^{\bfU,k}_{m'}\cup B_{{-m_1}}^{\bfX,k}(\bfZ)\subset
        \\&\subset
        B^{\bfU,k}_{m'}\cup B_{{-m}}^{\bfX,k}(\bfZ).
        \end{align*}
                
        Now, choose an affine open cover $\bfX=\bigcup_i \bfX_i$ s.t. each pair $(\bfX_i, \bfU_i:=\bfX_i\cap \bfZ)$ satisfies the conditions of the previous case. Choose simple rectifications on the $\bfX_i$ and arbitrary rectifications on the $\bfU_i$. WLOG assume that the rectifications on $\bfX$ and $\bfU$ are obtained from the rectifications on the $\bfX_i$ and the $\bfU_i$. By the previous case we can find $m_1$ s.t.
        \begin{equation}\label{lem:dev:eq4}
        B^{\bfX_i,k}_{m (s_\bfX+1)}\subset B^{\bfU_i,k}_{m'}\cup B^{\bfX_i,k}_{-m  (s_\bfX+1)}(\bfZ\cap \bfX_i).
        \end{equation}
        So we have 
        \begin{align*}        B^{\bfX,k}_m&=\bigcup_i (B^{\bfX_i,k}_{m (s_\bfX+1) } \cap  B^{\bfX_i,k}_m)\subset
        \\&\overset{\eqref{lem:dev:eq4}}{\subset}
        \bigcup_i ( B^{\bfU_i,k}_{m'}\cup  (B^{\bfX_i,k}_{-m(s_\bfX+1)}(\bfZ\cap \bfX_i)\cap B^{\bfX_i,k}_m))\subset
        \\&\overset{\text{Cor }\ref{cor:price.Z}}{\subset}
        \bigcup_i ( B^{\bfU_i,k}_{m'}\cup B^{\bfX,k}_{-m}(\bfZ))\overset{\text{Def }\ref{def: balls and measures}\eqref{def: balls and measures:a}}{=} B^{\bfU,k}_{m'}\cup B^{\bfX,k}_{-m}(\bfZ)
        \end{align*}
        }
        \item General case\\
        We will prove the statement by induction on $\dim \bfZ$.
        Let $\bfZ'$ be the singular locus of $\bfZ$. Let $\bfU'=\bfX  \smallsetminus \bfZ'$. Choose \typos{an arbitrary} rectification of $\bfU'$.
        Fix $m$.
        By the induction assumption  there exist $m_1>m$ such that for any $k$ we have 
        \begin{equation}\label{lem:dev:eq5}
        B^{\bfX,k}_m\subset B^{\bfU',k}_{m_1}\cup B^{\bfX,k}_{-m}(\bfZ').
        \end{equation}
        \typos{By \Cref{lem:cont.bnd}\eqref{lem:cont.bnd:S2}, applied to the embedding $\bfU'\to \bfX$ with its $m$ equal to $m_1$, there exists $m_0>m_1$ such that for any $k$ and any $z\in B^{\bfU',k}_{m_1}$ we have
        \begin{equation}\label{lem:dev:eq6}
        B^{\bfU',k}_{-m_0}(z)\subset B^{\bfX,k}_{-m_1}(z)\subset B^{\bfX,k}_{-m}(z).
        \end{equation}}
        By the previous case
        there exist $m_2>\typos{m_0}$ such that for any $k$ we have 
        \begin{equation}\label{lem:dev:eq7}
        B^{\bfU',k}_{\typos{m_0}}\subset B^{\bfU,k}_{m_2}\cup B^{\bfU',k}_{-{\typos{m_0}}}((\bfU'\cap \bfZ)).
        \end{equation}
        \typos{Note also that if $y\in B^{\bfU',k}_{-m_0}(z)\cap B^{\bfU',k}_{m_1}$ for some $z\in \bfU'\cap \bfZ$, then $z\in B^{\bfU',k}_{-m_0}(y)$ by \Cref{lem:tri}\eqref{lem:tri:sym}, hence $z\in B^{\bfU',k}_{m_1}$ by \Cref{lem:tri}\eqref{lem:tri:mon}, and therefore $y\in B^{\bfX,k}_{-m}(z)$ by \eqref{lem:dev:eq6}. Thus
        \begin{equation}\label{lem:dev:eq8}
        B^{\bfU',k}_{-m_0}((\bfU'\cap \bfZ))\cap B^{\bfU',k}_{m_1}\subset B^{\bfX,k}_{-m}((\bfU'\cap \bfZ)).
        \end{equation}}
        Take $m'=m_2$. We get
        \begin{align*}
        B^{\bfX,k}_m&\overset{\typos{\eqref{lem:dev:eq5}}}{\subset} B^{\bfU',k}_{m_1}\cup B^{\bfX,k}_{-m}(\bfZ')\typos{=}
        \\&\typos{=}        
        \typos{(B^{\bfU',k}_{m_0}\cap B^{\bfU',k}_{m_1})\cup B^{\bfX,k}_{-m}(\bfZ')\subset}
        \\&\overset{\typos{\eqref{lem:dev:eq7}}}{\subset}        
        \typos{((B^{\bfU,k}_{m_2}\cup B^{\bfU',k}_{-{m_0}}((\bfU'\cap \bfZ)))\cap B^{\bfU',k}_{m_1})\cup B^{\bfX,k}_{-m}(\bfZ')}
        \\&\subset        
        B^{\bfU,k}_{m_2}\cup \typos{(}B^{\bfU',k}_{-{\typos{m_0}}}((\bfU'\cap \bfZ))\typos{\cap B^{\bfU',k}_{m_1})}\cup B^{\bfX,k}_{-m}(\bfZ')
        \\&\overset{\typos{\eqref{lem:dev:eq8}}}{\subset}        
        B^{\bfU,k}_{m_2}\cup B^{\bfX,k}_{-{m}}((\bfU'\cap \bfZ))\cup B^{\bfX,k}_{-m}(\bfZ')
        \\&\subset        
        B^{\bfU,k}_{m_2}\cup B^{\bfX,k}_{-{m}}(\bfZ)
        \end{align*}
        as required.

    \end{enumerate}
\end{proof}

\begin{thm}\label[thm]{lem:bnd.on.sup}
    Let $\gamma: \bfX\to \bfY$ be  a submersion of \RamiB{$\mu$-rectified} varieties.  Then for any $m,m'\in \N$ there is $M\in \N$ such that for any $k$ 
    \RamiJ{and any $S\subset B_{m'}^{\bfX,k}$ satisfying $S=\bigcup_{x\in S} B_{-m}^{\bfX,k}(x)$}    
    we have  $$\mu_{m}^{\bfY,k}\cdot 1_{\RamiJ{\gamma(S)}} \RamiE{\leq} \ell^{kM}\gamma_*(\mu_{m'}^{\bfX,k}\RamiJ{1_S})$$
\end{thm}
\begin{proof}
$ $
    \begin{enumerate}[{Case} 1.]    
        \item \RamiJ{T}he \RamiB{$\mu$-rectification}s of $\bfX$ and of $\bfY$ are simple, $\gamma$  is \'etale and the form on $\bfX$ is the pullback of the form on $\bfY$:\\
        Take $M=1$.
        Let $\omega_\bfY$ be the form on $\bfY$ given by its \RamiB{$\mu$-rectification}. For $k\in\N$, let $\omega_{\bfY,k}:=(\omega_{\bfY})_{\F_{\ell^k}((t))}$ be the corresponding form on the extension of scalars $\bfY_{\F_{\ell^k}((t))}$.
        We have $$\gamma_*(\mu_{m'}^{\bfX,k}\RamiJ{\cdot 1_S})=f\cdot |\omega_{\bfY\RamiJ{,k}}|,$$ where $$f(y)=\#\{x\in \RamiJ{S}|\gamma(x)=y\}.$$
        Now, we have:
        \begin{align*}            
        \mu_{m}^{\bfY,k}\cdot 1_{\RamiJ{\gamma(S)}}&=1_{B_m^{\bfY,k}}\cdot 1_{\RamiJ{\gamma(S)}}\cdot |\omega_{\bfY,k}|=
        \\&=        
        1_{B_m^{\bfY,k}}\cdot 1_{\supp(f)}\cdot |\omega_{\bfY,k}|\RamiE{\leq}\ell^k\cdot f \cdot |\omega_{\bfY,k}|=        
        \ell^{kM}\gamma_*(\mu_{m'}^{\bfX,k}\RamiJ{\cdot 1_S}),
        \end{align*}
        as required.
        \item \RamiJ{T}he \RamiB{$\mu$-rectification}s of $\bfX$ and of $\bfY$ are simple and $\gamma$  is \'etale:\\
        Follows from the previous case and \Cref{lem:indep-rect}.        
        \item \RamiJ{T}he rectifications of $\bfX$ and of $\bfY$ are simple and  $\gamma=\gamma_1\circ \gamma_2$  where $\gamma_1:\bfY\times \bA^j\to \bfY$ is a projection and $\gamma_2$ is \'etale:\\
        Choose the \RamiB{$\mu$-rectification} of $\bfY\times \bA^j$ that is coming from the \RamiB{$\mu$-rectification} of $\bfY$ and the standard rectification of $\bA^j$. Let $\omega_{\bfY,k}$ be as above.
        
        Fix $m,m'\in \N$\RamiJ{, and let $S$ be as in the statement}. 
        By \Cref{lem:cont.bnd} (applied to the map $\gamma_2$) there exists $m_1$ such that $$\gamma_2(B_{m'}^{\bfX,k}) \subset B_{m_1}^{\bfY\times \A^j,k}.$$
        By the previous case we have 
        $M_1\in \N$ such that for any $k$ we have  
        \begin{equation*}
          \mu_{m_1}^{\bfY\times \A^j,k}\cdot 1_{\RamiJ{\gamma_2(S)}} \RamiE{\leq} \ell^{kM_1}(\gamma_2)_*(\mu_{m'}^{\bfX,k}\RamiJ{\cdot 1_S}).  
        \end{equation*}
        By \Cref{lem:eff.imp} (applied to the map $\gamma_2$ \RamiJ{with $\max(m,m')$}) we have $m_2$ such that for any $z\in \RamiJ{\gamma_2(S)}$ we have \begin{equation}\label{lem:cont.bnd:gam2opn}
        B_{-m_2}^{\bfY\times\A^j,k}(z)\subset \RamiJ{\gamma_2(S)}.
        \end{equation}        
        Take $M=M_1+\typos{j}m_2.$
        We have $$(\gamma_1)_*(\mu_{m_1}^{\bfY\times \A^j,k}\cdot 1_{\RamiJ{\gamma_2(S)}})=g\cdot |\omega_{\bfY,k}|$$ where 
        \begin{align*}
           g(y):&=Vol(\{v\in \F_{\ell^k}((t))^j|(y,v)\in B_{m_1}^{\bfY\times \A^j,k}\cap \gamma_2(\RamiJ{S})\})=
           \\&=
           Vol(\{v\in \F_{\ell^k}((t))^j|(y,v)\in \gamma_2(\RamiJ{S})\}).   
        \end{align*}
        Note that by \eqref{lem:cont.bnd:gam2opn} if $g(y)\neq 0$ then $$g(y)\geq Vol(B_{-m_2}^{\bA^{\typos{j}},k}(0)) = \ell^{-k\typos{j}m_2}.$$
        Thus
        \begin{align*}
        \mu_{m}^{\bfY,k}\cdot 1_{\RamiJ{\gamma(S)}}&= \mu_{m}^{\bfY,k}\cdot 1_{\supp(g)}\leq  \mu_{m}^{\bfY,k}\cdot g\ell^{k\typos{j}m_2}\leq |\omega_{\bfY,k}| \cdot g\ell^{k\typos{j}m_2}=
        \\&=
        (\gamma_1)_*(\mu_{m_1}^{\bfY\times \A^j,k}\cdot 1_{\RamiJ{\gamma_2(S)}})\ell^{k\typos{j}m_2}\RamiE{\leq}
        \\&\RamiE{\leq}
        (\gamma_1)_*(\ell^{kM_1}(\gamma_2)_*(\mu_{m'}^{\bfX,k}\RamiJ{\cdot 1_S}))\ell^{k\typos{j}m_2}=
        \\&
        = \ell^{kM_1+k\typos{j}m_2}\gamma_*(\mu_{m'}^{\bfX,k}\RamiJ{\cdot 1_S})=
         \ell^{kM}\gamma_*(\mu_{m'}^{\bfX,k}\RamiJ{\cdot 1_S})
        \end{align*}

    \item  The rectifications of $\bfX$ \RamiJ{and $\bfY$ are} simple:\\
    \RamiJ{
    Fix $m,m'\in \N$, WLOG $m'\geq 1$, \RamiJ{and let $S$ be as in the statement}.
Let $\omega$ be the form on $\bfX$.  We write
    $\omega_k:=(\omega)_{\F_{\ell^k}((t))}$.

    By \cite[\href{https://stacks.math.columbia.edu/tag/039P}{Tag 039P}]{SP} there
    is a finite open cover $\bfX=\bigcup_{i=1}^n\bfW_i$ by affine subvarieties such
    that each $\gamma|_{\bfW_i}$ satisfies the conditions of the previous case.
    Equip each $\bfW_i$ with a simple $\mu$-rectification whose form is
    $\omega|_{\bfW_i}$.

    {Let $\bfX'$ be the variety $\bfX$ equipped with the rectification given
    by this cover. By \Cref{lem:indep-rect}\eqref{lem:indep-rect:1a} there is $c$
    such that $B_{m'}^{\bfX,k}\subset B_c^{\bfX',k}$, that is}
    \begin{equation}\label{lem:bnd.on.sup:cover}
    B_{m'}^{\bfX,k}\overset{\text{Def }\ref{def: balls and measures}\eqref{def: balls and measures:a}}{\subset} \bigcup_{i=1}^n B_c^{\bfW_i,k} .
    \end{equation}

    {Put $S_i:=B_c^{\bfW_i,k}\cap S$.}
    By \Cref{lem:cont.bnd}\eqref{lem:cont.bnd:S2}, applied to the open embeddings
    $\bfW_i\to \bfX$, there is $\rho$ such that $B_{-\rho}^{\bfW_i,k}(x)\subset
    B_{-\RamiJ{m}}^{\bfX,k}(x)$ for any $x\in B_c^{\bfW_i,k}$. {Thus
    $S_i=\bigcup_{x\in S_i}B_{-\rho}^{\bfW_i,k}(x)$, by
    \Cref{lem:tri}\eqref{lem:tri:mon}.}

    \RamiJ{By the previous case, applied to $\gamma|_{\bfW_i}$, there is $M_i$ such that}
    \begin{equation}\label{lem:bnd.on.sup:ball}
    \mu_m^{\bfY,k}\cdot 1_{\gamma(S_i)}\leq \ell^{kM_i}\gamma_*(|\omega_k|\cdot 1_{S_i}).
    \end{equation}

    Take $M:=\max_i M_i+n$. {We get}
    \begin{align*}
    \mu_m^{\bfY,k}\cdot 1_{\RamiJ{\gamma(S)}}
    &\overset{\eqref{lem:bnd.on.sup:cover}}{\leq}
    \sum_{i=1}^n \mu_m^{\bfY,k}\cdot 1_{\gamma(S_i)}\\
    &\overset{\eqref{lem:bnd.on.sup:ball}}{\leq}
    \sum_{i=1}^n \ell^{kM_i}\gamma_*(|\omega_k|\cdot 1_{S_i})\\
    &\leq \ell^{kM}\gamma_*(|\omega_k|\cdot 1_{\RamiJ{S}})\\
    &\overset{\text{Def }\ref{def: balls and measures}}{=}
    \ell^{kM}\gamma_*(\mu_{m'}^{\bfX,k}\RamiJ{\cdot 1_S}).
    \end{align*}
}
\RamiJ{
    \item\label{lem:bnd.on.sup:simpY}  The rectification of $\bfY$ is simple:\\
    Follows from the previous case.
}
\RamiJ{
        \item $\gamma$ is an open embedding, the rectification of $\bfX$ is
        simple and $\bfX$ is one of the charts in the $\mu$-rectification of
        $\bfY$:\\
        Fix $m,m'\in \N$ and let $S$ be as in the statement. Let
        $\bfY=\bigcup_{\delta\in J}\bfV_\delta$ be the cover given by the
        rectification of $\bfY$, with forms $\eta_\delta$, and let $\delta_0\in J$
        be such that $\bfX=\bfV_{\delta_0}$. By \Cref{lem:trans}, for any $\delta$
        the variety $\bfV_\delta\cap \bfV_{\delta_0}$ is affine and closed in the
        product of the two ambient affine spaces, and $\eta_\delta/\eta_{\delta_0}$
        is an invertible regular function on it. By \Cref{lem:cont.bnd}\eqref{lem:cont.bnd:S1} there is $C$, independent of
        $k$ and of $\delta$, such that
        $|\eta_{\delta,k}|\leq \ell^{kC}|\eta_{\delta_0,k}|$ on $B_{m}^{\bfV_\delta,k}\cap B_{m'}^{\bfV_{\delta_0},k}$. Take
        $M:=C+|J|$. Since $\gamma$ is injective we have
        $\gamma_*(\mu_{m'}^{\bfX,k}\cdot 1_S)=|\eta_{\delta_0,k}|\cdot 1_S$ and
        $\gamma(S)=S$, so, summing over $\delta$ and using
        $S\subset B_{m'}^{\bfX,k}$,
        $$\mu_m^{\bfY,k}\cdot 1_{\gamma(S)}=
        \sum_{\delta\in J}|\eta_{\delta,k}|\cdot 1_{B_m^{\bfV_\delta,k}}\cdot 1_S
        \leq \ell^{kM}|\eta_{\delta_0,k}|\cdot 1_S=
        \ell^{kM}\gamma_*(\mu_{m'}^{\bfX,k}\cdot 1_S).$$       

        \item\label{lem:bnd.on.sup:open} $\gamma$ is an open embedding and the
        rectification of $\bfX$ is simple:\\
        Let $\bfY'$ be the variety $\bfY$ equipped with the $\mu$-rectification
        obtained by adding $\bfX$, with its embedding and form, to the charts of
        the $\mu$-rectification of $\bfY$. By \Cref{def: balls and measures} we
        have $\mu_m^{\bfY,k}\leq \mu_m^{\bfY',k}$, so the assertion follows from
        the previous case, applied to $\gamma:\bfX\to \bfY'$.
}
\RamiJ{
        \item  The rectification on $\bfX$ comes from rectifications on the
        preimages of the charts of the rectification on $\bfY$:\\
        Fix $m,m'\in \N$ and let $S$ be as in the statement. Let
        $\bfY=\bigcup_{\beta\in J}\bfV_\beta$ be the cover given by the
        rectification of $\bfY$ and let $\bfX_\beta:=\gamma^{-1}(\bfV_\beta)$, with
        the rectifications from which the one of $\bfX$ is obtained. By
        \Cref{def: balls and measures} we have
        \begin{equation}\label{lem:bnd.on.sup:split}
        \mu_{m'}^{\bfX,k}=\sum_{\beta\in J}\mu_{m'}^{\bfX_\beta,k},\qquad
        \mu_{m}^{\bfY,k}=\sum_{\beta\in J}\mu_{m}^{\bfV_\beta,k},\qquad
        B_{m'}^{\bfX,k}=\bigcup_{\beta\in J}B_{m'}^{\bfX_\beta,k}.
        \end{equation}

        Put $S_\beta:=B_{m'}^{\bfX_\beta,k}\cap S$, so that $S=\bigcup_\beta
        S_\beta$ by \eqref{lem:bnd.on.sup:split}. By
        \Cref{lem:cont.bnd}\eqref{lem:cont.bnd:S2}, applied to the open embeddings
        $\bfX_\beta\to \bfX$, there is $\rho$ such that
        $B_{-\rho}^{\bfX_\beta,k}(x)\subset B_{-m}^{\bfX,k}(x)$ for any
        $x\in B_{m'}^{\bfX_\beta,k}$; hence
        $S_\beta=\bigcup_{x\in S_\beta}B_{-\rho}^{\bfX_\beta,k}(x)$, by
        \Cref{lem:tri}\eqref{lem:tri:mon}. By
        \Cref{lem:cont.bnd}\eqref{lem:cont.bnd:S1} there is $m_0$ such that
        $\gamma(B_{m'}^{\bfX_\beta,k})\subset B_{m_0}^{\bfV_\beta,k}$ for all
        $\beta$.

        By Case \ref{lem:bnd.on.sup:simpY}, applied to $\gamma|_{\bfX_\beta}:\bfX_\beta\to \bfV_\beta$, there are $M_\beta$ such that
        \begin{equation}\label{lem:bnd.on.sup:gen}
        \mu_{m_0}^{\bfV_\beta,k}\cdot 1_{\gamma(S_\beta)}\leq
        \ell^{kM_\beta}\gamma_*(\mu_{m'}^{\bfX_\beta,k}\cdot 1_{S_\beta}).
        \end{equation}
        By Case \ref{lem:bnd.on.sup:open}, applied to the open
        embeddings $\bfV_\beta\to \bfY$,
        there exists $C$ s.t.
        $\mu_m^{\bfY,k}\cdot 1_{B_{m_0}^{\bfV_\beta,k}}\leq
        \ell^{kC}\mu_{m_0}^{\bfV_\beta,k}$ for all $\beta$. Take
        $M:=\max_\beta M_\beta+C$. Since $S_\beta\subset S$, we get
        \begin{align*}
        \mu_m^{\bfY,k}\cdot 1_{\gamma(S)}
        &\leq \sum_{\beta\in J} \mu_m^{\bfY,k}\cdot 1_{\gamma(S_\beta)}\\
        &\overset{\text{Case }\ref{lem:bnd.on.sup:open}}{\leq}
        \ell^{kC}\sum_{\beta\in J} \mu_{m_0}^{\bfV_\beta,k}\cdot 1_{\gamma(S_\beta)}\\
        &\overset{\eqref{lem:bnd.on.sup:gen}}{\leq}
        \ell^{kC}\sum_{\beta\in J} \ell^{kM_\beta}
        \gamma_*(\mu_{m'}^{\bfX_\beta,k}\cdot 1_{S_\beta})\\
        &\overset{\eqref{lem:bnd.on.sup:split}}{\leq}
        \ell^{kM}\gamma_*(\mu_{m'}^{\bfX,k}\cdot 1_{S}).
        \end{align*}

        \item  General case:\\
        Fix $m,m'\in \N$ and let $S$ be as in the statement. Let
        $\bfY=\bigcup_{\beta\in J}\bfV_\beta$ be the cover given by the
        rectification of $\bfY$, equip each $\gamma^{-1}(\bfV_\beta)$ with a
        $\mu$-rectification, and let $\bfX'$ be the variety $\bfX$ equipped with
        the $\mu$-rectification obtained from these. By
        \Cref{lem:indep-rect}\eqref{lem:indep-rect:1} there are $c$ and $\rho$ such
        that $B_{m'}^{\bfX,k}\subset B_c^{\bfX',k}$ and
        $B_{-\rho}^{\bfX',k}(x)\subset B_{-m}^{\bfX,k}(x)$ for any
        $x\in B_c^{\bfX',k}$. In particular $S\subset B_c^{\bfX',k}$ and
        $S=\bigcup_{x\in S}B_{-\rho}^{\bfX',k}(x)$, so by the previous case,
        applied to $\gamma:\bfX'\to \bfY$ with the parameters $m$, $c$ and $\rho$,
        there is $M_1$ such that
        $$\mu_m^{\bfY,k}\cdot 1_{\gamma(S)}\leq
        \ell^{kM_1}\gamma_*(\mu_c^{\bfX',k}\cdot 1_{S}).$$
        Let $\bfX=\bigcup_\alpha \bfU_\alpha$ be the cover given by the
        rectification of $\bfX$, each $\bfU_\alpha$ equipped with the corresponding
        simple $\mu$-rectification, and let $M_2$ be given by Case
        \ref{lem:bnd.on.sup:open}, applied to the open embeddings
        $\bfU_\alpha\to \bfX'$ with the parameters $c$ and $m'$. Every point of $S$
        lies in $B_{m'}^{\bfU_\alpha,k}$ for some $\alpha$, so
        $\mu_c^{\bfX',k}\cdot 1_{S}\leq \ell^{kM_2}\mu_{m'}^{\bfX,k}\cdot 1_{S}$,
        and the assertion follows with $M:=M_1+M_2$.
}
    \end{enumerate}
\end{proof}
\section{$m$-smooth functions and measures}\label{ssec:smoothfunme}
In this subsection we give a quantitative notion for smoothness of functions and measures, and prove that this notion behaves well under push. In particular we prove \Cref{thm: eff.smooth.of.push.of.measures }.
\begin{definition}
    Let $\bfX$ be a rectified variety. Let $m,k\in\N$. We say that 
    $f\in C^\infty_\RamiD{c}(B_\infty^{\bfX,k})$
    is $m$-smooth if for any 
    $x\in B_\infty^{\bfX,k}$
    the function $f|_{B_{-m}^{\bfX,k}(x)}$ 
    is constant.\index{$m$-smooth}
\end{definition}
\begin{lem}[\RamiD{Criteria} for  $m$-smoothness]\label{lem:sm.crit}
    Let $\bfX$ be a rectified variety. Let $m,k\in\N$. 
    \begin{enumerate}
        \item     \label{lem:sm.crit:1}
    Let 
    $f\in C^\infty_\RamiD{c}(B_\infty^{\bfX,k})$    be a \Rami{real valued} function 
    such that for any 
    $x\in B_\infty^{\bfX,k}$
    we have $$\min(f(B_{-m}^{\bfX,k}(x)))=f(x).$$
    Then $f$ is $m$-smooth.       
\item \label{lem:sm.crit:2}
    Let 
    $f\in C^\infty_\RamiD{c}(B_\infty^{\bfX,k})$    be a function 
    such that for any 
    $x\in \supp(f)$
    we have $f|_{B_{-m}^{\bfX,k}(x\RamiE{)}}$ is constant.
    Then $f$ is $m$-smooth.      
    \end{enumerate}    
\end{lem}
\begin{proof}
    Follows immediately from \Cref{lem:tri}\eqref{lem:tri:sym}.
\end{proof}
\begin{lemma}\label{lem:ext.by.0.sm} 
Let $\bfX$ be a rectified algebraic  variety and 
$\bfU\subset \bfX$ be an open subset. Fix a rectification on $\bfU$. Let 
$f\in C^\infty(B_m^{\bfU,k})$ be an $m$-smooth function. Then there is $m'$ such that $f\in C^\infty(B_{m'}^{\bfX,k})$ is an $m'$-smooth function. 
\end{lemma}
\begin{proof}
    Fix $m\in \N$.
    By \Cref{lem:cont.bnd}\eqref{lem:cont.bnd:S1} there exists $m_1>m$ such that for any $k\in \N$ we have $B_m^{\bfU,k}\subset B_{m_1}^{\bfX,k}$. 
    \RamiD{
    By \Cref{lem:eff.imp}  there exists $m'>m_1$ such that for any $k\in \N$ and any $x\in \supp(f) \subset  B_m^{\bfU,k}$ we have $B_{-m}^{\bfU,k}(x)\supset B_{-m'}^{\bfX,k}(x)$. By \Cref{lem:sm.crit}\eqref{lem:sm.crit:2} this implies the assertion.
    }    
\end{proof}
\begin{lem}\label{lem:poly:sm}
    Let $\bfX$ be a rectified algebraic variety and $f\in \cO^\times(\bfX)$. Then for any $m\in \N$ there exists $m'>m$ such that for any $k\in \N$ the function $|f| 1_{B^{\bfX,k}_m}$ on $B^{\bfX,k}_\infty$ is $m'$-smooth. 
\end{lem}
\begin{proof}
Consider $f$ as a function to $\A^1\smallsetminus 0$. By \Cref{lem:cont.bnd}\eqref{lem:cont.bnd:S1} there is $m_1>m$ such that  for any $k\in \N$ we have $$\max(\val(f(B_m^{\RamiD{\bfX},k}))) <m_1.$$ 
    By  \Cref{lem:cont.bnd}\eqref{lem:cont.bnd:S2}  we obtain that  there is $m_2>m_1$ such that  for any $k\in \N$
    and any $x\in B_m^{\bfX,k}$ we have $$\min(\val(f(B_{-m_2}^{\bfX,k}(x))-f(x)))> m_1.$$ 
    Take $m'=m_2$. We obtain that 
    for any $k\in \N$
    and any $x\in B_m^{\bfX,k}$ the function  $\val(f)|_{B_{-m'}^{\bfX,k}}$ is constant. This implies the assertion\RamiE{: for $x\in B_m^{\bfX,k}$ we have $B_{-m'}^{\bfX,k}(x)\subset B_m^{\bfX,k}$ by \Cref{lem:tri}\eqref{lem:tri:mon}, and for $x\notin B_m^{\bfX,k}$ the ball $B_{-m'}^{\bfX,k}(x)$ does not meet $B_m^{\bfX,k}$, by \Cref{lem:tri}\eqref{lem:tri:sym} and \Cref{lem:tri}\eqref{lem:tri:mon}}.
\end{proof}
\begin{lemma}\label{lem:sm.mes.def}
    Let $\gamma:\bfX_1 \to \bfX_2$ be an open embedding of  \RamiB{$\mu$-rectified} varieties. Then for any $m\in \N$ there \RamiH{are}  $\RamiH{m''>}m' \in \N$ such that for any $k\in \N$ \RamiH{and \EitanH{$0\leq m_0\leq m$}} there is an $\RamiH{m''}$-smooth function $f_k\in C^\infty_c(B_{m'}^{\bfX_2,k})$ such that:
    $$\gamma_*(\mu_{\RamiH{m_0}}^{\bfX_1,k})=f_k\cdot \mu_{m'}^{\bfX_2,k}$$
\end{lemma}
\begin{proof}$ $
    \begin{enumerate}[{Case} 1.]
    \item The \RamiB{$\mu$-rectification} of $\bfX_1$ is simple:\\
    Let $\bfX_2=\bigcup\bfU_i$ be the cover of $\bfX_2$. Let $\omega_i$ be the form on $\bfU_i$ and $\nu$ be the form on $\bfX_1$. 
\EitanH{
Fix \(m\). By \Cref{lem:cont.bnd}\eqref{lem:cont.bnd:S1}, choose \(m'>m\) such that
\(\gamma(B_m^{\RamiK{\bfX}_1,k})\subseteq B_{m'}^{\RamiK{\bfX}_2,k}\) for every \(k\).
Put \(W_i=\gamma^{-1}(\RamiK{\bfU}_i)\) and
\[
 g_i=\frac{(\gamma|_{\RamiK{\bfW}_i})^*\omega_i}{\nu|_{\RamiK{\bfW}_i}}
     \in\mathcal O(\RamiK{\bfW}_i)^\times.
\]
Give \(\RamiK{\bfW}_i\) the simple rectification induced by the given affine
embeddings of \(\RamiK{\bfX}_1\) and \(\RamiK{\bfU}_i\). 
In this rectification we have
\begin{equation}\label{eq:bal.in.int}
 B_{m'}^{\RamiK{\bfW}_i,k}
 = B_{m'}^{\RamiK{\bfX}_1,k}\cap\gamma^{-1}(B_{m'}^{\RamiK{\bfU}_i,k}).
\end{equation}

By \Cref{lem:poly:sm} applied on \(W_i\), the function
\(u_{i,k}=|g_i|\mathbf 1_{B_{m'}^{\RamiK{\bfW}_i,k}}\) is \(r_i\)-smooth
for some \(r_i\ge m'\) independent of \(k\). Its support lies in
\(B_{r_i}^{\RamiK{\bfW}_i,k}\). Hence, by \Cref{lem:ext.by.0.sm} for the open embedding
\(\RamiK{\bfW}_i\hookrightarrow \RamiK{\bfX}_1\), its extension by zero
\(\widetilde u_{i,k}\) is \(s_i\)-smooth on \(\RamiK{\bfX}_1\), with \(s_i\)
independent of \(k\).

Set \(f_{i,k}=\mathbf 1_{B_m^{\RamiK{\bfX}_1,k}}\widetilde u_{i,k}\).
By \eqref{eq:bal.in.int} and \(m'\ge m\), this function equals
\(|g_i|\) on \(B_m^{\RamiK{\bfX}_1,k}\cap\gamma^{-1}(B_{m'}^{\RamiK{\bfU}_i,k})\)
and vanishes elsewhere. Since the source rectification is simple,
\(\mathbf 1_{B_m^{\RamiK{\bfX}_1,k}}\) is \(1\)-smooth. Taking
\(m_2>\max_i\{s_i,m'\}\), we conclude that all the \(f_{i,k}\)
are \(m_2\)-smooth, uniformly in \(k\).

}


    Let $h_{k}=\sum_i f_{i,k}$. We obtain:
    \begin{itemize}
        \item $\gamma^*(\mu_{\RamiH{m'}}^{\bfX_2,k})1_{B_m^{\bfX_1,k}}=h_{k} \mu_m^{\bfX_1,k}.$
        \item $h_{k}$ is $m_2$-smooth. 
    \end{itemize}
Thus,
$$(\mu_{\RamiH{m'}}^{\bfX_2,k})1_{\gamma(B_m^{\bfX_1,k})}=(\mu_{\RamiH{m'}}^{\bfX_2,k})\gamma_*(1_{B_m^{\bfX_1,k}})=\gamma_*(\gamma^*(\mu_{\RamiH{m'}}^{\bfX_2,k})1_{B_m^{\bfX_1,k}})=\gamma_*(h_{k} \mu_m^{\bfX_1,k})=\gamma_*(h_{k})\gamma_*( \mu_m^{\bfX_1,k}).$$
    Note that 
    \begin{align*}        
    \supp(h_k)&=\bigcup \supp(f_{i,k})=\bigcup (B_{m}^{\bfX_1,k}\cap \gamma^{-1}(B_{\RamiH{m'}}^{\bfU_i,k}))=B_{m}^{\bfX_1,k} \cap \bigcup \gamma^{-1}(B_{\RamiH{m'}}^{\bfU_i,k})=
    \\&=
    B_{m}^{\bfX_1,k} \cap \gamma^{-1}(B_{\RamiH{m'}}^{\bfX_2,k})=B_{m}^{\bfX_1,k} \end{align*}

    Let $$f_k(x)=
    \begin{cases}
			\gamma_*(h_k)^{-1}(x), & \text{if $x\in \supp(\gamma_*(h_k))$ }\\
            0, & \text{otherwise}
		 \end{cases}$$
         We get 
$$(\mu_{\RamiH{m'}}^{\bfX_2,k})f_k=(\mu_{\RamiH{m'}}^{\bfX_2,k})
1_{\gamma(B_m^{\bfX_1,k})}f_k=f_k\gamma_*(h_{k})\gamma_*( \mu_m^{\bfX_1,k})=1_{\gamma(B_m^{\bfX_1,k})}\gamma_*( \mu_m^{\bfX_1,k})=\gamma_*( \mu_m^{\bfX_1,k}).$$
So it remains to show that there is $\RamiH{m''}$ such that for any $k$ \RamiH{and any \EitanH{$0\leq m_0\leq m$},} the function $f_k \RamiH{1_{\gamma(B_{m_0}^{\bfX_1,k})}}$ is an $\RamiH{m''}$-smooth function on $B^{\bfX_2,k}_\infty$. This follows from the above using \Cref{lem:ext.by.0.sm}.
         
    \item \RamiE{T}he general case:\\
    \RamiE{Let $\bfX_1=\bigcup \bfU_i$ be the cover given by its rectification. The assertion f}ollows from the previous case \RamiE{applied to the pairs $(\bfU_i,\bfX_2)$}.
\end{enumerate}    

\end{proof}

\begin{lem}\label{lem:push.fun}
    Let $\gamma:\bfX_1 \to \bfX_2$ be an \'etale morphism of rectified varieties. Then for any $m\in \N$ there is  $m' \in \N$ such that for any      
    $k\in \N$ and any $m$-smooth function $g\in C^\infty_c(B_{m}^{\bfX_1,k})$,     
    the function \Rami{ $f:=\gamma_*(g)$ on $B_{\infty}^{\bfX_2,k}$ defined by $$f(y)=\sum_{x\in (\gamma^{-1}(y))(\F_{\ell^k}((t)))\cap B_{m}^{\bfX_1,k}}g(x)$$ is $m'$-smooth.}
\end{lem}
\begin{proof}
\Rami{WLOG assume that $g$ is real and non-negatively valued.}
    We will use the \RamiD{criterion for $m$-smoothness (\Cref{lem:sm.crit}\eqref{lem:sm.crit:1}}). 
    Fix $m$.
   
    By \Cref{lem:et.mono} there is $m_1>m$ such that
     \begin{equation}\label{lem:push.fun:eq1}
         \forall k\in \N, x\in B^{\bfX_1,k}_m \text{ the map } \gamma|_{B^{\bfX_1,k}_{-m_1}(x)} \text{ is a monomorphism}
     \end{equation}

    By \Cref{lem:eff.imp} there exists $m'>m_1$ such that 
    for any $k$ and any $x\in B^{\bfX_1,k}_m$ we have 
    \begin{equation}\label{lem:push.fun:1}            
    \gamma(B^{\bfX,k}_{-m_1}(x))\supset B^{\bfY,k}_{-m'}(\gamma(x)).
    \end{equation}

    \Rami{Fix $k\in \N$. For} $y\in B_{\infty}^{\bfX_2,k}$, 
    \Rami{denote 
    $$\fF_y:= (\gamma^{-1}(y))(\F_{\ell^k}((t))).$$
    Fix $y\in B_{\infty}^{\bfX_2,k}$, and denote}    
    $$\{x_1,\dots,x_N\}:=\Rami{\fF_y}\cap B_{m}^{\bfX_1,k}.$$
    By \eqref{lem:push.fun:eq1} and \Cref{lem:tri}\eqref{lem:tri:sym}, \Rami{for every \EitanH{$i \ne j\in \{1,\dots, N\}$,} we have:}
        \begin{equation}\label{lem:push.fun:2}       B_{-m_1}^{\bfX_1,k}(x_i)\cap B_{-m_1}^{\bfX_1,k}(x_j)=\emptyset.
\end{equation}
        
    Let $y'\in B_{-m'}^{\bfX_2,k}(y)$. 
    We \Rami{obtain 
    \begin{align*}        
    \gamma_*(g)(y')&=\sum_{x\in \fF_{y'}}g(x)\overset{\eqref{lem:push.fun:1}}{\geq}\sum_{i=1}^{N} \sum_{x\in \fF_{y'}\cap B_{-m_1}^{\bfX_1,k}(x_i)}g(x)\overset{\eqref{lem:push.fun:2}}{\geq}
    \\&\geq
    \sum_{i=1}^N \min(g(B_{-m_1}^{\bfX_1,k}(x_i)))=\sum_{i=1}^N g(x_i)=\gamma_*(g)(y).
    \end{align*}
    }
    \RamiD{By \Cref{lem:sm.crit}\eqref{lem:sm.crit:1}}  this implies the assertion.       
\end{proof}

\begin{thm}\label[thm]{lem:push.mes.sm}
    Let $\gamma:\bfX_1 \to \bfX_2$ be a  smooth map  of  \RamiB{$\mu$-rectified} varieties. Then for any $m\in \N$ there \RamiH{are}  $\RamiH{m''>}m' \in \N$ such that for any $k\in \N$, \RamiH{\EitanH{$0\leq m_0\leq m$},} and any $m$-smooth function $g\in C^\infty_\RamiD{c}(B_{\infty}^{\bfX_1,k})$ there is an $\RamiH{m''}$-smooth function $f\in C^\infty_c(B_{m'}^{\bfX_2,k})$ such that:
    $$\gamma_*(g\mu_{\RamiH{m_0}}^{\bfX_1,k})=f\cdot \mu_{m'}^{\bfX_2,k}.$$
\end{thm}

\begin{proof}
$ $
\begin{enumerate}[{Case} 1.]    
    \item\label{lem:push.mes.sm:c1} $\gamma$ is an open embedding:\\
    Follows from \Cref{lem:ext.by.0.sm}  and \Cref{lem:sm.mes.def}.
\item\label{lem:push.mes.sm:c2} $\gamma$ is \'etale, both rectifications are simple and the form on $\bfX_1$ is the pullback of the form on $\bfX_2$:\\
Follows from \RamiH{\Cref{lem:ext.by.0.sm}  and} \Cref{lem:push.fun}.
\item\label{lem:push.mes.sm:c3} $\gamma$ is \'etale, and both rectifications are simple:\\
Follows from the previous 2 cases.
\item\label{lem:push.mes.sm:c4} $\gamma$ is \'etale, and the rectification of $\bfX_2$ is simple:\\
Follows from the previous case.
\item\label{lem:push.mes.sm:c5} $\gamma$ is \'etale and can be written as a composition: $\bfX_1\to\bfX_3\to \bfX_2$ where  the rectification of $\bfX_3$ is simple and $\bfX_3\to \bfX_2$ is an open embedding:\\
Follows from the previous case and Case \ref{lem:push.mes.sm:c1}.
\item\label{lem:push.mes.sm:c6} $\gamma$ is \'etale and the cover of $\bfX_1$ is obtained \DimaA{from} covers of the preimages of the covering sets on $\bfX_2$:\\
\RamiE{Let $\bfX_2=\bigcup \bfU_i$  be the covering given by the rectification. Let $\bfV_i:=\gamma^{-1}(\bfU_i)$.  Equip $\bfV_i$ with the $\mu$-rectification corresponding to the $\mu$-rectification on $\bfX_{\RamiH{1}}$ . Let $m\in \N$.
By \Cref{lem:sm.crit}\eqref{lem:sm.crit:2} and \Cref{cor:price.no.af} we can find $m_1$ s.t. for any $k,i$ and any $m$-smooth function 
$f\in C_c^{\infty}(B_{\infty}^{\bfX_\RamiH{1},k})$ 
the function  $f\cdot 1_{B_{m}^{\bfV_\RamiH{i},k}}$ is   $m_1$-smooth function on $B_{\infty}^{\bfV_i,k}$. 

By the previous step we can find $\RamiH{m''>}m'$ s.t. for any $i,k$\RamiH{, any \EitanH{$0\leq m_0\leq m$}, and any}  $m_1$-smooth function $g$ on $B_{\infty}^{\bfV_i,k}$  
 there is an $\RamiH{m''}$-smooth function $f\in C^\infty_c(B_{m'}^{\bfX_2,k})$ such that:
    $$(\gamma|_{\bfV_i})_*(g\mu_{\RamiH{m_0}}^{\bfV_i,k})=f\cdot \mu_{m'}^{\bfX_2,k}.$$

Now let  $k\in \N$\RamiH{, \EitanH{$0\leq m_0\leq m$},} and let  $g\in C^\infty(B_{\infty}^{\bfV_i,k})$ be an $m$-smooth function. We have 
$$g\mu_{\RamiH{m_0}}^{\bfX_1,k}=\sum  g\cdot  \RamiH{\mu_{m_0}^{\bfV_i,k}}$$

By the above, the function $g_i:=g\cdot  1_{B_{\RamiH{m_0}}^{\bfV_i,k}}$ is an $m_1$-smooth function on $\bfV_i$. Thus, by the above 
$(\gamma|_{\bfV_i})_*(g_i\mu_{\RamiH{m_0}}^{\bfV_i,k})=f_i\cdot \mu_{m'}^{\bfX_2,k}$ where $f_i\in C^\infty_c(B_{m'}^{\bfX_2,k})$ is $\RamiH{m''}$-smooth. 
We obtain that $$\gamma_*(g\mu_{\RamiH{m_0}}^{\bfX_1,k})=\gamma_*(\sum  g\cdot  1_{B_{\RamiH{m_0}}^{\bfV_i,k}}\mu_{\RamiH{m_0}}^{\bfV_i,k})= \sum (\gamma|_{\bfV_i})_* (g\cdot  1_{B_{\RamiH{m_0}}^{\bfV_i,k}}\mu_{\RamiH{m_0}}^{\bfV_i,k})=\sum f_i\cdot \mu_{m'}^{\bfX_2,k}$$
is $\RamiH{m''}$-smooth.
}

\item\label{lem:push.mes.sm:c7} $\gamma$ is \'etale:\\
Follows from the previous case  and Case \ref{lem:push.mes.sm:c1}.
\item\label{lem:push.mes.sm:c8} $\bfX_1=\bfX_2\times \A^n$, $\gamma$ is the projection, the rectifications on $\bfX_i$ are simple, and the rectification of $\bfX_1$ is obtained from the rectification of $\bfX_2$ in the natural way:\\
In this case we can take $m'=m$ \RamiJ{and $m''=m+1$}. \Rami{The assertion} follows from the fact that for 2 points $x_1,x_2\in B_{\infty}^{\bfX_2,k}$ with $x_2\in B_{\RamiJ{-m}}^{\bfX_2,k}(x_1)$ we have $j_1^*(g)=j_2^*(g)$  for any $m$-smooth $g\in C^\infty_\RamiD{c}(B_{\infty}^{\bfX_\Rami{1},k})$, where  $j_i:\A^n\to \bfX_\Rami{1}$ are given by $j_i(a)=(x_i,a)$. 
\item\label{lem:push.mes.sm:c9} $\bfX_1=\bfX_2\times \A^n$, $\gamma$ is the projection, and $\bfX_2$ is affine:\\
Follows from the previous case and Case 1.
\item\label{lem:push.mes.sm:c10} The general case:\\
Follows from the previous case Case \ref{lem:push.mes.sm:c7}, Case \ref{lem:push.mes.sm:c1} and the structure theorem for smooth maps 
(\cite[\href{https://stacks.math.columbia.edu/tag/039Q}{Theorem 039Q}]{SP}).
\end{enumerate}    
\end{proof}

\section{Effectively surjective maps}\label{sec:sur}
In this section we introduced a version of surjectivity of a map between algebraic varieties.
We complete the proof of \Cref{thm: Eff. bounds on push of smooth measures},
and prove \Cref{thm: crit. eff. surj.} - a criterion for effective surjectivity (See \Cref{prop:crit-effsurj} below).

\begin{defn} \label[defn]{def:effsur}
Let $\gamma:\bfX\to \bfY$ be a map of rectified varieties. We say that $\gamma$ is effectively surjective\index{effectively surjective} iff  for any $m \Rami{\in \N}$ there is $\Rami{m'}  \in \N $ such that for every $k\in \N$ we have $$\gamma(B^{\bfX,k}_{m'})\supset B^{\bfY,k}_m.$$
\end{defn}
From \Cref{lem:indep-rect} 
we obtain:
\begin{lemma}
    The property of a map $\gamma:\bfX\to \bfY$ being effectively surjective does not depend of the rectifications of the varieties $\bfX$ and $\bfY$.
\end{lemma}
\begin{thm}\label[thm]{prop:crit-effsurj}
        Let $\gamma:\bfX\to \bfY$ be a smooth map of
        \RamiK{rectified}
        algebraic varieties that is  onto on the level of points over any  field. Then $\gamma$ is effectively surjective.
\end{thm}
For this we will need some preparations.
\begin{lemma}\label{lem:onto-start}
    Let $\gamma:\bfX\to \bfY$ be a map of   algebraic \RamiB{varieties.} 
    Assume that 
       $\gamma$ is onto on the level of points for any field \RamiB{(that contains $\F_\ell$)}.
       Then,
       $\bfY$ admits a stratification  such that $\gamma$ admits a section for each strata.
\end{lemma}
\RamiB{
\begin{remark}
This is  a standard result which is valid over any field. For completeness, we include its proof here.     
\end{remark}
}
\begin{proof}
We will prove the statement
by Noetherian Induction. So it is enough to show that $\gamma$ admits a section on some nonempty open set $\bfV\subset \bfY$. Without loss of generality we  can assume that $\bfY$ is irreducible and affine.
\begin{enumerate}[{Case} 1:]
    \item $\dim(\bfY)=0$\\
    obvious.
    \item  $\dim(\bfY)>0$\\
Let $K$ be the field of rational functions on $\bfY$.  \Rami{By the assumption,}  $\gamma(K):\bfX(K)\to \bfY(K)$ is onto. We have a canonical point $y \in \bfY(K)$. This gives us a  point $x\in \bfX(K)$ such that $y=\gamma(x)$. So  we get a commutative diagram:
$$ 
\begin{tikzcd}
     & \bfX \dar["\gamma"] \\
    Spec(K)\rar["y"] \urar["x"]  & \bfY
\end{tikzcd}
$$ 
We have an affine open set $\bfU\subset \bfX$ such that $x\in \bfU(K)$. In other words we get a diagram
$$ 
\begin{tikzcd}
     \bfU\rar[""]& \bfX \dar["\gamma"] \\
    Spec(K)\rar["y"] \uar["x"]  & \bfY
\end{tikzcd}
$$ 
This gives us a map $O_\bfU(\bfU)\to K$. Since $O_\bfU(\bfU)$ is  finitely generated over $\F_\ell$, the image of this map lies inside $f^{-1}O_\bfY(\bfY)$ for some $f\in O_\bfY(\bfY)$. Let $\bfV:=\bfY_f$ be the non vanishing locus of $f$. We get a commutative diagram:
$$ 
\begin{tikzcd}
     \bfU\arrow["",rr]& & \bfX \arrow["\gamma",dd] \\
      &\bfV \arrow["",dr] \arrow["",ul] & \\
    Spec(K)\arrow["y",rr] \arrow["x",uu]  \arrow["",ur] & & \bfY
\end{tikzcd}
$$ 
This gives the requested section.
\end{enumerate}
\end{proof}

The following follows immediately from \Cref{lem:indep-rect}:
\begin{lemma}\label{lem:sect-effsurj}
    Let $\gamma:\bfX\to \bfY$ be a map of  algebraic varieties defined over $\F_\ell$ that admits  a section. Then $\gamma$ is effectively surjective.
\end{lemma}

\begin{cor}\label[cor]{cor:glue-effsurj}
    Let $\gamma:\bfX\to \bfY$ be a smooth map of \RamiK{rectified} algebraic varieties defined over $\bF_\ell$.
    Let $\bfU\subset \bfY$ be open and $\bfZ:=\bfY\smallsetminus \bfU$. Assume that $\bfZ$ is smooth. Assume also that $\gamma|_{\gamma^{-1}(\bfZ)}:\gamma^{-1}(\bfZ)\to \bfZ$ and $\gamma|_{\gamma^{-1}(\bfU)}:\gamma^{-1}(\bfU)\to \bfU$ are effectively surjective. Then so is $\gamma$.    
\end{cor}
\begin{proof}
    Fix an integer $m$. Choose \typos{an arbitrary} rectification of $\gamma^{-1}(\bfU)$.
    By \Cref{cor:cls.embd} there exists $m_1>m$  such that for any $k$ we have 
    \begin{equation}\label{cor:glue-effsurj:eq1}
    B^{\bfZ,k}_{m_1} \supset  B^{\bfY,k}_m\cap \RamiD{B^{\bfZ,k}_\infty}.
    \end{equation}
    By the assumption there exists $m_2>m_1$  such that for any $k\in \N$ we have 
    \begin{equation}\label{cor:glue-effsurj:eq2}
    \gamma(B^{\gamma^{-1}(\bfZ),k}_{m_2})\supset B^{\bfZ,k}_{m_1}.
    \end{equation}
    By the \Cref{lem:cont.bnd} there exists $m_3>m_2$  such that for any $k$ we have 
    \begin{equation}\label{cor:glue-effsurj:eq3}
    B^{\bfX,k}_{m_3}\supset B^{\gamma^{-1}(\bfZ),k}_{m_2}.
    \end{equation}
    By the \Cref{lem:eff.imp} there exists $m_4>m_3$  such that for any $k$ and for any $x\in B^{\bfX,k}_{m_3}$ we have 
    \begin{equation}\label{cor:glue-effsurj:eq4}
    \gamma(B^{\bfX,k}_{-m_3}(x))\supset B^{\bfY,k}_{-m_4}(\gamma(x)).
    \end{equation}
    By the \Cref{lem:dev} there exists $m_5>m_4$  such that for any $k$ we have     
    \begin{equation}\label{cor:glue-effsurj:eq5}
    B^{\bfY,k}_{-m_4}(\bfZ)\cup B^{\bfU,k}_{m_5}\supset B^{\bfY,k}_{m_4}.
    \end{equation}
    By the assumption there exists  $m_6>m_5$  such that for any $k$ we have 
    \begin{equation}\label{cor:glue-effsurj:eq6}
    \gamma(B^{\gamma^{-1}(\bfU),k}_{m_6})\supset B^{\bfU,k}_{m_5}.
\end{equation}
\typos{By \Cref{lem:cont.bnd}\eqref{lem:cont.bnd:S1} there exists $m_7>m_6$ such that for any $k$ we have
\begin{equation}\label{cor:glue-effsurj:eqU}
B^{\bfX,k}_{m_7}\supset B^{\gamma^{-1}(\bfU),k}_{m_6}.
\end{equation}}

By  \Cref{lem:tri}\RamiD{(\ref{lem:tri:mon},\ref{lem:tri:sym})} for any 2 integers $a,b\in\N$ \RamiD{we have}:
\begin{equation}\label{cor:glue-effsurj:eq7}
\bigcup_{x\in B^{\bfY,k}_a\cap 
\RamiD{B^{\bfZ,k}_\infty}}B^{\bfY,k}_{-b}(x) = B^{\bfY,k}_a \cap B^{\bfY,k}_{-b}(\bfZ)
\end{equation}
\RamiD{and}
\begin{equation}\label{cor:glue-effsurj:eq8}
\bigcup_{x\in B^{\bfX,k}_a\cap \RamiD{B^{\gamma^{-1}(\bfZ),k}_\infty}}B^{\RamiJ{\bfX},k}_{-b}(x) = B^{\bfX,k}_a \cap B^{\RamiJ{\bfX},k}_{-b}(\gamma^{-1}(\bfZ)).
\end{equation}
    Take $m'=\typos{m_7}$.
    For any $k\in \N$ we have 
    \begin{align*}        \gamma(B^{\bfX,k}_{m'})&=\gamma(B^{\bfX,k}_{\typos{m_7}})
    \RamiD{\overset{\cref{lem:tri}\eqref{lem:tri:mon}\typos{,\eqref{cor:glue-effsurj:eqU}}}{\supset}}
    \gamma((B^{\bfX,k}_{-m_3}(\gamma^{-1}(\bfZ))\cap B^{\bfX,k}_{m_3}) \cup B^{\gamma^{-1}(\bfU),k}_{m_6})
        =\\&=
        \gamma(B^{\bfX,k}_{-m_3}(\gamma^{-1}(\bfZ))\cap B^{\bfX,k}_{m_3}) \cup \gamma(B^{\gamma^{-1}(\bfU),k}_{m_6})
        \supset\\&        
        \RamiD{\overset{(\ref{cor:glue-effsurj:eq8},\ref{cor:glue-effsurj:eq6})}{\supset}}
        \gamma\left(\bigcup_{x\in B^{\bfX,k}_{m_3}\cap \RamiD{B^{\gamma^{-1}(\bfZ),k}_\infty}}B^{\bfX,k}_{-m_3}(x)\right) \cup B^{\bfU,k}_{m_5}
        =\\&=
        \bigcup_{x\in B^{\bfX,k}_{m_3}\cap \RamiD{B^{\gamma^{-1}(\bfZ),k}_\infty}}\gamma(B^{\bfX,k}_{-m_3}(x)) \cup B^{\bfU,k}_{m_5}
        \supset\\&
        \RamiD{\overset{\eqref{cor:glue-effsurj:eq3}}{\supset}}        
        \bigcup_{x\in B^{\gamma^{-1}(\bfZ),k}_{m_2}}\gamma(B^{\bfX,k}_{-m_3}(x)) \cup B^{\bfU,k}_{m_5}
        \supset\\&
        \RamiD{\overset{\eqref{cor:glue-effsurj:eq4}}{\supset}}
        \bigcup_{x\in B^{\gamma^{-1}(\bfZ),k}_{m_2}}B^{\bfY,k}_{-m_4}(\gamma(x)) \cup B^{\bfU,k}_{m_5}
        \RamiD{=}\\&\RamiD{=}
        \bigcup_{x\in \gamma(B^{\gamma^{-1}(\bfZ),k}_{m_2})}B^{\bfY,k}_{-m_4}(x) \cup B^{\bfU,k}_{m_5}
        \supset\\&
        \RamiD{\overset{\eqref{cor:glue-effsurj:eq2}}{\supset}}
        \bigcup_{x\in B^{\bfZ,k}_{m_1}}B^{\bfY,k}_{-m_4}(x) \cup B^{\bfU,k}_{m_5}
        \supset\\&
        \RamiD{\overset{\eqref{cor:glue-effsurj:eq1}}{\supset}}
        \bigcup_{x\in B^{\bfY,k}_m\cap \RamiD{B^{\bfZ,k}_\infty}}B^{\bfY,k}_{-m_4}(x) \cup B^{\bfU,k}_{m_5}
        \supset\\&
        \RamiD{\overset{\eqref{cor:glue-effsurj:eq7}}{\supset}}
        (B^{\bfY,k}_m \cap B^{\bfY,k}_{-m_4}(\bfZ)) \cup B^{\bfU,k}_{m_5}
        \supset\\&\supset        
        B^{\bfY,k}_m \cap (B^{\bfY,k}_{-m_4}(\bfZ)) \cup B^{\bfU,k}_{m_5})
        \supset\\&
        \RamiD{\overset{\eqref{cor:glue-effsurj:eq5}}{\supset}}
        B^{\bfY,k}_m \cap B^{\bfY,k}_{m_4}=B^{\bfY,k}_m. 
    \end{align*} 
\end{proof}

\begin{proof}[Proof of \Cref{prop:crit-effsurj}]
By \Cref{lem:onto-start}, there is a stratification $\bfY=\bigcup_\alpha \bfY_{\alpha}$ such that $\gamma$ admits a section for each strata. \RamiE{Refine this stratification s.t. each $\bfY_\alpha$ is smooth.} By \Cref{lem:sect-effsurj} the maps $\gamma:\gamma^{-1}(\bfY_{\alpha}) \to \bfY_{\alpha}$ are effectively surjective. 
We will show that 
$\gamma:\bfX \to \bfY$ is effectively surjective by induction on the number of strata. The base follows from \Cref{lem:sect-effsurj}. For the induction step let $\bfY_0$ be a closed stratum. Let $\bfY'=\bfY\smallsetminus\bfY_0$.  By the induction hypothesis, the map $\gamma^{-1}(\bfY')\to \bfY'$ is effectively surjective. By \Cref{lem:sect-effsurj} the map $\gamma^{-1}(\bfY_0)\to \bfY_0$ is effectively surjective. So, by \Cref{cor:glue-effsurj}, $\gamma$ is effectively surjective. 
\end{proof}
\Cref{lem:bnd.on.sup} gives us the following:

\begin{cor}\label[cor]{cor:push.bnd.above}    
    Let $\gamma: \bfX\to \bfY$ be  a submersion of \RamiB{$\mu$-rectified} varieties. Assume that $\gamma$ is effectively surjective. Then for any $m$ there is $m'$ such that for any $k\in \N$  we have $$\mu_{m}^{\bfY,k} < \ell^{km'}\gamma_*(\mu_{m'}^{\bfX,k}).$$
\end{cor}
\begin{proof}
Fix $m$. 
    Since $\gamma$ is effectively surjective there is $m_1$ such that for every $k\in \bN$ we have $$\gamma(B^{\bfX,k}_{m_1})\supset B^{\bfY,k}_m.$$
    By \Cref{lem:bnd.on.sup}, there is $M$ such that for any $k$ we have  
    $$\mu_{m}^{\bfY,k}\cdot 1_{\supp(\gamma_*(\mu_{m_1}^{\bfX,k}))} < \ell^{kM}\gamma_*(\mu_{m_1}^{\bfX,k}).$$
    Take $m'=m_1+M$. For any $k\in \bN$ we have  
    $$\mu_{m}^{\bfY,k}=\mu_{m}^{\bfY,k}\cdot 1_{B^{\bfY,k}_m} \leq \mu_{m}^{\bfY,k}\cdot 1_{\gamma(B_{m_1}^{\bfX,k})}=\mu_{m}^{\bfY,k}\cdot 1_{\supp(\gamma_*(\mu_{m_1}^{\bfX,k}))} < \ell^{kM}\gamma_*(\mu_{m_1}^{\bfX,k})<\ell^{km'}\gamma_*(\mu_{m'}^{\bfX,k}).$$
\end{proof}

\begingroup
  \let\clearpage\relax
  \let\cleardoublepage\relax 
  \printindex
\endgroup

\bibliographystyle{alpha}
\bibliography{Ramibib}
\end{document}